\documentclass[12pt]{amsart}
\usepackage[margin=1in]{geometry}

\usepackage{url,amsmath,amssymb,mathrsfs}
\usepackage{enumitem}
\usepackage{commath}
\usepackage{physics}
\usepackage{nicefrac}
\usepackage{amsthm}
\usepackage{xfrac}
\usepackage[all]{xy}
\usepackage{graphicx}
\usepackage[normalem]{ulem}
\graphicspath{ {./images/} }

\newtheorem{theorem}{Theorem}[section]
\newtheorem{lemma}[theorem]{Lemma}
\newtheorem{proposition}[theorem]{Proposition}
\newtheorem{corollary}[theorem]{Corollary}

\newtheorem{innercustomgeneric}{\customgenericname}
\providecommand{\customgenericname}{}
\newcommand{\newcustomtheorem}[2]{%
	\newenvironment{#1}[1]
	{%
		\renewcommand\customgenericname{#2}%
		\renewcommand\theinnercustomgeneric{##1}%
		\innercustomgeneric
	}
	{\endinnercustomgeneric}
}
\newcustomtheorem{customthm}{Theorem}
\newcustomtheorem{customlemma}{Lemma}

\theoremstyle{definition}
\newtheorem{example}[theorem]{Example}
\newtheorem{definition}[theorem]{Definition}

\newtheorem*{remark}{Remark}

\newcommand{\bZ}{\mathbb Z}
\newcommand{\bQ}{\mathbb Q}
\newcommand{\bN}{\mathbb N}
\newcommand{\bNo}{{{\mathbb N}_0}}

\newcommand{\Rbar}{{\overline{R}}}

\newcommand{\Inv}{\textup{Inv}}
\newcommand{\Prin}{\textup{Prin}}
\newcommand{\Cl}{\textup{Cl}}

\newcommand{\modulo}[1]{\textup{ (mod }#1)}

\newcommand{\Irr}{\textup{Irr}}
\newcommand{\term}[1]{\textbf{\textup{#1}}}

\newcommand{\quot}[2]{\large\sfrac{#1}{#2}\normalsize}

\newcommand{\Nil}{\textup{Nil}}

\title[Associated Polynomials and Power Series]{(Locally) Associated Subrings in Polynomial and Power Series Extensions}
\author{Grant Moles\and Joseph Swanson}

\begin{document}

\begin{abstract}
    The associated, ideal-preserving, and locally associated properties of subrings, first formally defined in 2024, give a way of understanding the multiplicative structure of a subring given information about the larger ring. In this paper, we establish notation and preliminary results on a generalized type of polynomial and power series rings that is often used in the construction of counterexamples in the field of commutative algebra. We then provide sets of necessary and sufficient conditions under which such a polynomial or power series ring may be (locally) associated in a larger such ring. This allows for the production of several informative examples, as well as a better understanding of the circumstances under which the ring of formal power series $R[[x]]$ over an order $R$ in a number field may be half-factorial.
\end{abstract}

\maketitle

\section{Introduction}

In the study of factorization theory, there is perhaps no context that is more well understood and which has enjoyed such lasting historical interest than that of the rings of integers in algebraic number fields. Indeed, it was in this context (specifically in cyclotomic number fields) that Ernst Kummer demonstrated that unique factorization into irreducible elements fails in general, rebuffing Gabriel Lam\'e's now-infamous failed proof of Fermat's Last Theorem. Classically, Kummer went on to define ``ideal numbers" (the precursor to our modern understanding of an ideal in a ring) and the ideal class group, which allowed for the restoration of unique factorization in the set of ideals. In the latter half of the $20^{th}$ century, the ideal class group provided exactly the tool needed for an explosion in the understanding of factorization in rings of algebraic integers. Beginning with Carlitz in 1960 (\cite{carlitz}) and culminating with Valenza and Narkiewich in the 1990s (\cite{valenza} and \cite{narkiewicz}, respectively), this led to a full characterization of the elasticity of any ring of algebraic integers using the Davenport constant.

Since these beginnings, factorization theory has expanded to far more rings and monoids and explored far more properties (see, for instance, \cite{andersonandersonzafrullah}, \cite{ChapmanFactorizationTheory}, \cite{GeroldingerFactorizationSurvey}, \cite{coykendallpruferdomains}). One particularly interesting case that has received widespread attention is that of orders within algebraic number fields. Halter-Koch in \cite{halter-koch} provided a characterization of which orders in a quadratic number field are half-factorial; this was later generalized beyond the quadratic case by Rago in \cite{rago}. Investigation into the elasticity (among other factorization properties) of orders has been a fruitful direction of inquiry (see, for instance, \cite{choi2024class}, \cite{KettingerMoles2025elasticity}, \cite{SRelativeDavenport}). Perhaps most telling is a common thread that these investigations all seem to have in common: in one way or another, they leverage what is known about the full ring of integers in which the order is contained to give information about the order itself.

It was precisely this observation which led to the definition of three types of subrings (or equivalently, three types of ring extensions) in \cite{dissertation} and \cite{subringrelations}: associated, ideal-preserving, and locally associated subrings. The immediate advantages of studying these relations seemed to primarily lie in understanding the relationship between an order in a number field and the full ring of algebraic integers. However, it was shown in \cite{radicalconductor} that this utility also extended to the rings of formal power series over these rings. In particular, they helped to determine the elasticity of orders and their rings of formal power series, including when such rings may (or may not) be half-factorial. These subring relations will be defined and discussed more thoroughly in Section 2 below.

Another area of factorization theory (and commutative algebra in general) which has received a great deal of historical attention is the investigation of rings of polynomials and rings of formal power series. Generally, these questions tend to follow the following format: given information about a ring $R$, what properties can we determine about $R[x]$ or $R[[x]]$? (For instance, see \cite{gilerpolynomials}, \cite{coykendallextensions}, \cite{coykendallpolynomial}, \cite{radicalconductor}.) However, when searching for unexpected or ``bad" behavior of ring properties, attention is often given to subrings in which the ring of available coefficients is allowed to change (in particular, to expand) with higher powers of the variable. In \cite{anderson}, for example, rings of the form $K+xL[x]$, where $K\subsetneq L$ are fields (i.e., the subring of $L[x]$ whose constant coefficients are restricted to $K$), are shown to give one of the most readily-accessible examples of HFDs which are not UFDs. Similarly, $\bZ+x\bQ[x]$ gives a readily-accessible example of an integral domain which fails to be atomic. Constructions of this form can be found throughout the literature; for instance, see \cite{coykendallA+XBX}, \cite{changA+XBX}, \cite{gottiascentofquasi}.

In \cite{subringrelations}, it was shown that a similar construction, in particular the ring $\bZ+x\bZ+x^2\bQ[x]$, produced unexpected behavior with respect to the locally associated property (this behavior is explored more fully in Example \ref{leading example} below). Partial results regarding when a standard ring of polynomials or power series might be a locally associated subring of a larger polynomial or power series ring were produced in \cite[Theorems 4.1-4.3]{subringrelations}, though they failed to provide full characterizations of these properties and primarily served to generalize findings from the case of orders in a number field. In pursuit of a better understanding of both the associated and locally associated properties of subrings and these polynomial-like ring constructions, we investigate when extensions of such rings are associated or locally associated. We will also explore the consequences of these findings and use them to draw conclusions and settle conjectures about more familiar cases. The major results of this paper consist of at least partial characterizations of when these properties are observed in such ring extensions and can be found in Theorems \ref{associated polynomials}, \ref{la polynomial}, \ref{associated ps sufficient}, \ref{associated ps necessary}, and \ref{la power series}.

To avoid confusion on the part of the reader, we specify some of the notation and conventions we will be using throughout this paper:
\begin{itemize}
    \item $\bN$ refers to the set of natural numbers $\{1,2,3,\dots\}$.
    \item $\bN_0$ refers to the set of whole numbers $\{0,1,2,3,\dots\}$.
    \item A \term{ring} is not assumed to be commutative or to have identity; when those properties apply, they will be specified.
    \item For a commutative ring with identity $R$, $U(R)$ refers to the \term{group of units} in $R$.
    \item For an integral domain $R$, $\Rbar$ refers to the \term{integral closure} of $R$ in its quotient field.
    \item For a commutative ring $R$ and a multiplicatively closed subset $S\subseteq R$ containing no zero divisors, $R_S:=S^{-1}R=\{s^{-1}r|r\in R,s\in S\}$ is the \term{localization} of $R$ at $S$.
    \item For subrings $R$ and $T$ of the same commutative ring, the \term{conductor} from $T$ into $R$ is the set $(R:T):=\{r\in R|rT\subseteq R\}$; that is, $(R:T)$ is the set of all elements in $R$ which conduct $T$ into $R$. Note that we do not require $R$ to be a subring of $T$, nor do we require that elements of $(R:T)$ exist in $T$. Note that $(R:T)$ is an ideal of $R$, and when $(R:T)\subseteq T$, it is an ideal of $T$ as well.
    \item For any integral domain $R$, we denote by $\Cl(R)$ the group of invertible fractional ideals $\Inv(R)$ modulo the subgroup of principal fractional ideals $\Prin(R)$, called the \term{ideal class group} (or \term{Picard group}) of $R$. The ideal class containing $J\in \Inv(R)$ is denoted $[J]$.
    \item For an atomic domain $R$, we denote by $\rho(R)$ the \term{elasticity} of $R$, defined in \cite{valenza}.
\end{itemize}

The remainder of this paper is laid out as follows. Section 2 more thoroughly explores the associated, ideal-preserving, and locally associated properties of subrings and generalizes a previously-known characterization of locally associated subrings. Section 3 investigates polynomial and power series extensions in which the rings of coefficients are allowed to change with higher powers of $x$. This includes defining a notation for such rings and investigating the units and conductor ideals in such rings and ring extensions. Finally, Sections 4 and 5 apply the findings from Section 3 to determine when extensions of these polynomial rings and power series rings, respectively, may be associated or locally associated.

\section{(Locally) Associated Subrings}

For the convenience of the reader, we recall the definitions of associated and locally associated subrings from \cite{subringrelations}. Although we will use it far less prominently in this paper, we also include the definition of an ideal-preserving subring for completeness and to have ready access to the terminology.

\begin{definition}
    \label{associated}
    Let $T$ be a commutative ring with identity and $R\subseteq T$ a subring. We say that $R$ is an \term{associated subring} of $T$ if either of the following equivalent conditions holds:
    \begin{enumerate}
        \item $T=R\cdot U(T)$; that is, for any $t\in T$, there exist $r\in R$ and $u\in U(T)$ such that $t=ru$.
        \item For any $t\in T$, there exists some $u\in U(T)$ such that $tu\in R$.
    \end{enumerate}
\end{definition}

\begin{definition}
    \label{locally associated}
    Let $T$ be a commutative ring with identity, $R\subseteq T$ a subring with identity, and $I=(R:T)$. We say that $R$ is a \term{locally associated subring} of $T$ if any of the following equivalent conditions hold:
    \begin{enumerate}
        \item For any $t\in T$ comaximal to $I$ (i.e., $tT+I=T$), there exists $r\in R$ comaximal to $I$ (i.e., $rR+I=R$) and $u\in U(T)$ such that $t=ru$.
        \item Every coset in $\quot{U(\sfrac{T}{I})}{U(\sfrac{R}{I})}$ contains a unit in $T$; that is, for any $t+I\in U(\sfrac{T}{I})$, there exists some $r+I\in U(\sfrac{R}{I})$ and $\beta\in I$ such that $tr+\beta\in U(T)$.
        \item $\quot{U(T)}{U(R)}\cong \quot{U(\sfrac{T}{I})}{U(\sfrac{R}{I})}$ via the isomorphism $\phi:\quot{U(T)}{U(R)}\to \quot{U(\sfrac{T}{I})}{U(\sfrac{R}{I})}$ defined by $\phi(u\cdot U(R))=(u+I)U(\sfrac{R}{I})$
    \end{enumerate}
\end{definition}

\begin{definition}
    \label{ideal-preserving}
    Let $T$ be a commutative ring and $R\subseteq T$ a subring. We say that $R$ is an \term{ideal-preserving subring} of $T$ if, for any pair of ideals $J_1$ and $J_2$ of $T$ with $J_1\nsubseteq J_2$, $R\cap J_1\nsubseteq J_2$.
\end{definition}

Since these relationships seem to be at the height of their power when $R$ is an order in an algebraic number field and $T=\Rbar$ is the full ring of integers, we often say in this case that $R$ is an associated (similarly, locally associated or ideal preserving) order.

The concept of an associated subring seems to have originated in the study of factorization properties of orders in an algebraic number field. In particular, \cite[Theorem 6]{halter-koch} included this property as one of the necessary and sufficient conditions for an order in a quadratic number field to be half-factorial. Later, \cite[Theorem 1.1]{rago} showed that this property was necessary for an order in any number field to be half-factorial. More generally, \cite[Corollary 3.5]{radicalconductor} showed that if $R$ is an order such that $\rho(R)=\rho(\Rbar)$, then $R$ must be a locally associated order; the same paper also used the associated condition as one of a list of sufficient conditions for $\rho(R)=\rho(\Rbar)$ and for $\rho(R[[x]])=\rho(\Rbar[[x]])$, though it was later shown in \cite[Proposition 4.3]{SRelativeDavenport} that the associated condition is not necessary for $\rho(R)=\rho(\Rbar)$.

One of the earliest observations in the study of these subring relations is that any associated order must also be both ideal-preserving and locally associated; in fact, \cite[Corollary 4.13]{subringrelations} shows that the converse holds as well. In addition, if $R\subseteq S\subseteq \Rbar$ and $R$ is an associated (similarly, locally associated or ideal-preserving) order, then $S$ must share this same property. Although these are straightforward observations in the case of orders in a number field, \cite[Theorem 2.11]{subringrelations} demonstrates that the locally associated property in a general ring extension may be less well-behaved.

\begin{example}
    \label{leading example}
    Let $R=\bZ[x]$, $S=\bZ+x\bZ+x^2\bQ[x]$ and $T=\bQ[x]$. That is, $S$ is the subring of $T$ whose constant and linear coefficients are restricted to $\bZ$. Since any element of $T$ can be multiplied by an appropriate integer to clear denominators, $R$ (and by extension, $S$) is an associated subring of $T$. Moreover, $(R:T)=\{0\}$, so characterization (3) in the definition of a locally associated subring makes it clear that $R$ is a locally associated subring of $T$. On the other hand, $I:=(S:T)=(x^2)$, and we may observe that $t:=2+3x\in T$ is comaximal to $I$. However, the only elements of $S$ which are comaximal to $I$ must have constant term $\pm 1$, so no unit $u\in U(T)=\bQ\backslash\{0\}$ exists such that $ut$ is an element of $S$ comaximal to $I$.

    This shows that not only may a subring $S\subseteq T$ exist which is associated without being locally associated, it is also possible to construct a tower $R\subseteq S\subseteq T$ in which $R$ is a locally associated subring of $T$ but $S$ is not.
\end{example}

One will note that, as mentioned in the introduction, this readily-available example which demonstrates a failure of expected ``nice" properties comes from the fact that the coefficient rings in $\bZ+x\bZ+x^2\bQ[x]$ are allowed to change with the power of $x$. We will explore these rings more thoroughly in the following section.

Before concluding our discussion of these subring relations, we take a moment to refine an equivalent characterization of locally associated subrings. Using an exact sequence from \cite[Proposition 12.9]{neukirch}, it was shown in \cite[Theorem 3.4]{subringrelations} that, among other similar characterizations, an order $R$ in a number field is locally associated if and only if $\Cl(R)\cong \Cl(\Rbar)$. We generalize this characterization by first generalizing the exact sequence, which is of independent interest.

\begin{lemma}
    \label{comaximal intersection}
    Let $T$ be an integral domain, $R$ a subring with identity, and $I$ an ideal of $T$ which is contained in $R$. If $J$ is an ideal of $T$ which is comaximal to $I$ (i.e., $J+I=T$), then $R\cap J$ is an ideal of $R$ which is comaximal to $I$ (i.e., $R\cap J+I=R$).
\end{lemma}

\begin{proof}
    Let $J$ be an ideal of $T$ comaximal to $I$. It is immediate that $R\cap J$ is an ideal of $R$. Since $J$ is comaximal to $I$, there exist $\alpha\in J$ and $\beta\in I$ such that $\alpha+\beta=1$. Since $\beta\in I\subseteq R$, then $\alpha=1-\beta\in R\cap J$. Then $1\in R\cap J+I$, so $R\cap J$ is comaximal to $I$.
\end{proof}

\begin{lemma}
    \label{comaximal to AI}
    Let $R$ be an integral domain and $I$ and $J$ ideals of $R$. Then for any ideal $A$ of $R$ which is comaximal to $J$, $J+I=J+AI$. In particular, $J$ is comaximal to $I$ if and only if $J$ is comaximal to $AI$.
\end{lemma}

\begin{proof}
    First, note that since $AI\subseteq I$, then $J+AI\subseteq J+I$ trivially. We will show the reverse inclusion. Let $\alpha\in J+I$; that is, there exist $\beta_1\in J$ and $\gamma_1\in I$ such that $\alpha=\beta_1+\gamma_1$. Furthermore, since $A$ is comaximal to $J$, there exist $\beta_2\in J$ and $\gamma_2\in A$ such that $\beta_2+\gamma_2=1$. Then
    $$\alpha=(\beta_1+\gamma_1)(\beta_2+\gamma_2)=(\beta_1\beta_2+\beta_1\gamma_2+\beta_2\gamma_1)+\gamma_1\gamma_2\in J+AI.$$
    Thus, $J+I=J+AI$.
\end{proof}

\begin{lemma}
    \label{going up or going down}
    Let $T$ be an integral domain, $R$ a subring with identity, and $I=(R:T)$. If $J$ is an ideal of $R$ which is comaximal to $I$, then $JT\cap R=J$. If $A$ is an ideal of $T$ that is comaximal to $I$, then $(R\cap A)T=A$.
\end{lemma}

\begin{proof}
    Let $J$ be an ideal of $R$ comaximal which is comaximal to $I$. Then trivially, $J\subseteq JT\cap R$; we need to show the reverse inclusion. Since $J$ is comaximal to $I$, then there exists some $\alpha\in J$ and $\beta\in I$ such that $\alpha+\beta=1$. Now let $r=j_1t_1+\dots+j_kt_k\in JT\cap R$, with each $j_i\in J$ and $t_i\in T$. Then
    $$r=r(\alpha+\beta)=\alpha r+j_1(t_1\beta)+\dots+j_k(t_k\beta).$$
    Since $\alpha\in J$ and $r\in R$, then $\alpha r\in J$. Since $j_i\in J$ and $t_i\beta\in R$, then $j_i(t_i\beta)\in J$ for every $1\leq i\leq k$. Then $r\in J$, so $JT\cap R=J$.

    Now let $A$ be an ideal of $T$ comaximal to $I$. Trivially, $(R\cap A)T\subseteq A$; we need to show the reverse inclusion. Since $A$ is comaximal to $I$, there exists some $\alpha\in A$ and $\beta\in I$ such that $\alpha+\beta=1$; note that $\alpha=1-\beta\in R$. Now for any $a\in A$, note that $a=a(\alpha+\beta)=a\alpha+a\beta$. Since $\alpha\in R\cap A$ and $a\in T$, then $a\alpha\in (R\cap A)T$. Since $a\in A$ and $\beta\in I$, then $a\beta\in R\cap A\subseteq (R\cap A)T$. Then $a\in (R\cap A)T$, so $(R\cap A)T=A$.
\end{proof}

\begin{lemma}
    \label{still invertible}
    Let $T$ be an integral domain with quotient field $K$, $R$ a subring with identity, and $I=(R:T)$. If $J$ is an invertible integral ideal of $R$, then $JT$ is an invertible integral ideal of $T$. If $A$ is an invertible integral ideal of $T$ which is comaximal to $I$, then $R\cap A$ is an invertible integral ideal of $R$ which is comaximal to $I$.
\end{lemma}

\begin{proof}
    Let $J$ be an invertible integral ideal of $R$, i.e., $JJ^{-1}=R$, where $J^{-1}=\{\alpha\in K|\alpha J\subseteq R\}$. For any $\alpha\in J^{-1}$, note that $\alpha J\subseteq R$, so $\alpha(JT)\subseteq RT=T$. Then $\alpha\in (JT)^{-1}$. Since $J\subseteq JT$, then $1\in JJ^{-1}\subseteq (JT)(JT)^{-1}$, so $(JT)(JT)^{-1}=T$. Then $JT$ is an invertible integral ideal of $T$.

    Now let $A$ be an invertible integral ideal of $T$ which is comaximal to $I$, i.e., $AA^{-1}=T$ and there exist $\alpha_1\in A$ and $\beta_1\in I$ such that $\alpha_1+\beta_1=1$. Lemma \ref{comaximal intersection} tells us immediately that $R\cap A$ is an ideal of $R$ which is comaximal to $I$; we need to show that $R\cap A$ is invertible. Since $\beta_1$ is comaximal to $A$, then Lemma \ref{comaximal to AI} tells us that $A$ is also comaximal to $\beta_1I$. Then there exist $\alpha_2\in A$ and $\beta_2\in I$ such that $\alpha_2+\beta_1\beta_2=1$. Since $AA^{-1}=T$, then $1=a_1b_1+\dots+a_kb_k$ for some $a_i\in A$ and $b_i\in A^{-1}$, $1\leq i\leq k$. Then note that 
    $$1=(a_1b_1+\dots+a_kb_k)(\alpha_2+\beta_1\beta_2)=\alpha_2+(a_1\beta_1)(b_1\beta_2)+\dots+(a_k\beta_1)(b_k\beta_2).$$
    Since $\alpha_2=1-\beta_1\beta_2\in R\cap A$ and $1\in R\subseteq (R\cap A)^{-1}$, then $\alpha_2\in (R\cap A)(R\cap A)^{-1}$. For each $1\leq i\leq k$, note that $a_i\beta_1\in R\cap A$. Furthermore, since $b_1\in A^{-1}$, then for any $a\in A$, $a(b_i\beta_2)=(ab_i)\beta_2\in R$. Then each $b_i\beta_2\in (R\cap A)^{-1}$, so $(a_i\beta_1)(b_i\beta_2)\in (R\cap A)(R\cap A)^{-1}$. Thus, $1\in (R\cap A)(R\cap A)^{-1}$. Therefore, $(R\cap A)(R\cap A)^{-1}=R$, meaning that $R\cap A$ is an invertible integral ideal of $R$ which is comaximal to $I$.
\end{proof}

\begin{theorem}
    \label{exact sequence}
    Let $T$ be an integral domain, $R$ a subring with identity, and $I=(R:T)$. Consider the following sequence:
    $$1\xrightarrow{\iota} U(R)\xrightarrow{\phi} U(T)\oplus U(\sfrac{R}{I})\xrightarrow{\psi} U(\sfrac{T}{I})\xrightarrow{\sigma} \Cl(R)\xrightarrow{\tau}\Cl(T)\xrightarrow{\pi}1$$
    with maps $\phi$, $\psi$, and $\sigma$ defined as follows:
    \begin{alignat*}{2}
        &\phi:U(R)\to U(T)\oplus U(\sfrac{R}{I}),&&\quad\quad\phi(v)=(v,v^{-1}+I);\\
        &\psi:U(T)\oplus U(\sfrac{R}{I})\to U(\sfrac{T}{I}),&&\quad\quad\psi(u,r+I)=ur+I;\\
        &\sigma:U(\sfrac{T}{I})\to \Cl(R),&&\quad\quad\sigma(t+I)=[tT\cap R];\\
        &\tau:\Cl(R)\to\Cl(T),&&\quad\quad\tau([J])=[JT].
    \end{alignat*}
    Then $\phi$, $\psi$, $\sigma$, and $\tau$ are homomorphisms. Moreover, the sequence $(\iota,\phi,\psi,\sigma)$ is exact. The sequence $(\iota,\phi,\psi,\sigma,\tau)$ is exact if and only if every ideal class in $\ker(\tau)$ (i.e., any class $[J]\in\Cl(R)$ such that $JT$ is principal) contains an ideal class representative $J$ which is an integral ideal of $R$ comaximal to $I$. If, in addition to the previous condition, every ideal class in $\Cl(T)$ contains an ideal class representative $J$ which is an integral ideal of $T$ comaximal to $I$, then the sequence $(\iota,\phi,\psi,\sigma,\tau,\pi)$ is exact.
\end{theorem}

\begin{proof}
    To start, we will show that the maps $\phi$, $\psi$, $\sigma$, and $\tau$ are (well-defined) homomorphisms. Then, we will show exactness.

    First, note that for any $v\in U(R)$, $v^{-1}\in R\subseteq T$. Then $v\in U(T)$ as well. Furthermore, $(v+I)(v^{-1}+I)=1+I$, so $v^{-1}+I\in U(\sfrac{R}{I})$. Then $\phi$ is a well-defined function. Moreover, note that for any $v_1,v_2\in U(R)$,
    $$\phi(v_1v_2)=(v_1v_2,(v_1v_2)^{-1}+I)=(v_1,v_1^{-1}+I)(v_2,v_2^{-1}+I)=\phi(v_1)\cdot \phi(v_2).$$
    Thus, $\phi$ is a homomorphism.

    Now let $u\in U(T)$ and $r+I\in U(\sfrac{R}{I})$. Then for some $s\in R$, $(r+I)(s+I)=1+I$. For this choice of $s$, note that $(ur+I)(u^{-1}s+I)=1+I$ in $\sfrac{T}{I}$, so $ur+I\in U(\sfrac{T}{I})$. Moreover, if $u\in U(T)$ and $r_1+I=r_2+I\in U(\sfrac{R}{I})$, then $r_1-r_2\in I$. Therefore, $ur_1-ur_2=u(r_1-r_2)\in I$, so $ur_1+I=ur_2+I$. Thus, $\psi$ is a well-defined function. To see that $\psi$ is a homomorphism, note that for $u_1,u_2\in U(T)$ and $r_1+I,r_2+I\in U(\sfrac{R}{I})$,
    $$\psi(u_1u_2,r_1r_2+I)=u_1u_2r_1r_2+I=(u_1r_1+I)(u_2r_2+I)=\psi(u_1,r_1+I)\cdot\psi(u_2,r_2+I).$$

    Now let $t+I\in U(\sfrac{T}{I})$; then $J:=tT\cap R$ is an ideal of $R$. Since $tT$ is a principal (and thus invertible) ideal of $T$ comaximal to $I$, then Lemma \ref{still invertible} tells us that $J$ is an invertible ideal of $R$ comaximal to $I$. Then $[J]=[tT\cap R]\in \Cl(R)$. Letting $t_1+I=t_2+I\in U(\sfrac{T}{I})$, we note that $t_1-t_2\in I$. Then for any $s\in T$ such that $t_1s\in R$, $t_2s=t_1s+(t_2-t_1)s\in R$. By symmetry of $t_1$ and $t_2$, this tells us that for any $s\in T$, $t_1s\in R$ if and only if $t_2s\in R$. Then $t_1T\cap R=\frac{t_1}{t_2}(t_2T\cap R)$, so $[t_1T\cap R]=[t_2T\cap R]$ in $\Cl(R)$. Thus, $\sigma$ is a well-defined function.

    To show that $\sigma$ is a homomorphism, let $t_1+I,t_2+I\in U(\sfrac{T}{I})$. Trivially, $(t_1T\cap R)(t_2T\cap R)\subseteq t_1t_2T\cap R$. To show the reverse inclusion, let $\alpha\in T$ such that $t_1t_2\alpha\in R$. Since $t_1$ and $t_2$ are units modulo $I$, there exist $s_1,s_2\in T$ and $\beta_1,\beta_2\in I$ such that $t_1s_1+\beta_1=t_2s_2+\beta_2=1$ (in particular, $t_1s_1,t_2s_2\in R$). Then:
    \begin{align*}
        t_1t_2\alpha&=t_1t_2\alpha(t_1s_1+\beta_1)(t_2s_2+\beta_2)\\&=[t_1s_1][t_2s_2\cdot t_1t_2\alpha]+[t_1s_1][t_2(t_1\alpha\beta_2)]+[t_1(t_2\alpha\beta_1)][t_2s_2]+[t_1(\alpha\beta_1)][t_2\beta_2].
    \end{align*}
    Since each addend on the right-hand side of this identity is an element of $(t_1T\cap R)(t_2T\cap R)$, then $t_1t_2\alpha\in (t_1T\cap R)(t_2T\cap R)$. Thus, $$\sigma(t_1t_2+I)=[t_1t_2T\cap R]=[t_1T\cap R][t_2T\cap R]=\sigma(t_1+I)\cdot \sigma(t_2+I),$$
    so $\sigma$ is a homomorphism.

    Now let $[J]\in \Cl(R)$; without loss of generality, we may assume that $J$ is an integral ideal of $R$. By Lemma \ref{still invertible}, $JT$ is an invertible ideal of $T$, so $[JT]\in \Cl(T)$. Furthermore, if $[J_1]=[J_2]$, i.e., if there exist nonzero $r_1,r_2\in R$ such that $J_1=\frac{r_1}{r_2}J_2$, then $J_1T=\frac{r_1}{r_2}J_2T$. Thus, $\tau([J_1])=[J_1T]=[J_2T]=\tau([J_2])$, so $\tau$ is a well-defined function. To see that $\tau$ is in fact a homomorphism, note that for any $[J_1],[J_2]\in\Cl(R)$, $$\tau([J_1J_2])=[J_1J_2T]=[J_1T][J_2T]=\tau([J_1])\cdot\tau([J_2]).$$

    We now need to show exactness of the sequence $(\iota,\phi,\psi,\sigma)$. To do so, we need to show the following:
    \begin{align*}
        \{1\}&=\ker(\phi);\\
        \phi(U(R))&=\ker(\psi);\\
        \psi(U(T)\oplus U(\sfrac{R}{I}))&=\ker(\sigma).
    \end{align*}

    First, we will show that $\ker(\phi)$ is trivial; that is, that $\phi$ is injective. Let $u\in U(R)$ such that $\phi(u)=(u,u^{-1}+I)=(1,1+I)$. Then trivially $u=1$, so $\ker(\phi)=\{1\}$.

    We now want to show that $\phi(U(R))=\ker(\psi)$. First, let $v\in U(R)$ and note that $\psi(\phi(v))=\psi(v,v^{-1}+I)=vv^{-1}+I=1+I$. Then $\phi(v)\in \ker(\psi)$, so $\phi(U(R))\subseteq\ker(\psi)$. For the reverse inclusion, let $(t,r+I)\in \ker(\psi)$, i.e., $t\in U(T)$ and $r+I\in U(\sfrac{R}{I})$ such that $\psi(t,r+I)=tr+I=1+I$. Since $r+I\in U(\sfrac{R}{I})$, there exist $s\in R$ and $\beta_1\in I$ such that $sr+\beta_1=1$; since $tr+I=1+I$, there exists $\beta_2\in I$ such that $tr+\beta_2=1$. Then $t-s=(t-s)(sr+\beta_1)=(tr-sr)s+(t-s)\beta_1=(\beta_1-\beta_2)s+(t-s)\beta_1\in I$. Since $s\in R$ and $t-s\in I\subseteq R$, this means that $t\in R$. Moreover, $tr+\beta_2=1$ and $t^{-1}\in T$, so $t^{-1}=r+t^{-1}\beta_2\in R$; then $t\in U(R)$. Furthermore, $t^{-1}+I=(t+I)^{-1}=r+I$, so $t\in U(R)$ such that $(t,r+I)=(t,t^{-1}+I)=\phi(t)\in \phi(U(R))$. Thus, $\ker(\psi)\subseteq \phi(U(R))$.

    We will now show that $\psi(U(T)\oplus U(\sfrac{R}{I}))=\ker(\sigma)$. First, let $u\in U(T)$ and $r+I\in U(\sfrac{R}{I})$; then $\sigma(\psi(u,r+I))=\sigma(ur+I)=[urT\cap R]=[rT\cap R]$. Since $rT=(rR)T$, with $rR$ comaximal to $I$, then Lemma \ref{going up or going down} tells us that $(rR)T\cap R=rR$. Then $\sigma(\psi(u,r+I))$ is the principal (identity) class in $\Cl(R)$, so $\psi(U(T)\oplus U(\sfrac{R}{I}))\subseteq \ker(\sigma)$. To show the reverse inclusion, suppose that $t+I\in \ker(\sigma)$, i.e., $t+I\in U(\sfrac{T}{I})$ and $tT\cap R$ is a principal ideal in $R$. Then there exists $r\in R$ such that $tT\cap R=rR$. Since $t+I\in U(\sfrac{T}{I})$, there exists some $s\in T$ and $\beta\in I$ such that $st+\beta=1$. Since $st=1-\beta\in tT\cap R$, then $st=r\gamma$ for some $\gamma\in R$, so $r\gamma+\beta=1$. Then $r+I\in U(\sfrac{R}{I})$. Since $t$ is comaximal to $I$, Lemma \ref{going up or going down} tells us that $tT=(tT\cap R)T=(rR)T=rT$. Then $t$ and $r$ are associates in $T$, i.e., there exists $u\in U(T)$ such that $t=ur$, so $t+I=ur+I=\psi(u,r+I)\in \psi(U(T)\oplus U(\sfrac{R}{I}))$. Therefore, $\ker(\sigma)\subseteq \psi(U(T)\oplus U(\sfrac{R}{I}))$, so we conclude that $(\iota,\phi,\psi,\sigma)$ is an exact sequence.

    Now assume that every ideal class in $\ker(\tau)$ contains an ideal class representative $J$ which is an integral ideal of $R$ comaximal to $I$. Under this condition, we need to show that $\sigma(U(\sfrac{T}{I}))=\ker(\tau)$. Note that for any $t+I\in U(\sfrac{T}{I})$, the ideal $tT$ is comaximal to $I$. Then Lemma \ref{going up or going down} tells us that $\tau(\sigma(t+I))=\tau([tT\cap R])=[(tT\cap R)T]=[tT]$, the principal (identity) class in $\Cl(T)$. Then $\sigma(U(\sfrac{T}{I}))\subseteq\ker(\tau)$. To show the reverse inclusion, suppose that $[J]\in \ker(\tau)$, i.e., $JT=tT$ is principal. By our assumption, we can choose $J$ to be an integral ideal of $R$ comaximal to $I$. Then $R=J+I$, so $T=RT=(J+I)T=JT+I$. Then $JT=tT$ is comaximal to $I$, so $t+I\in U(\sfrac{T}{I})$. Then by Lemma \ref{going up or going down}, $tT\cap R=JT\cap R=J$, so $\sigma(t+I)=[tT\cap R]=[J]$. Thus, $[J]\in \sigma(U(\sfrac{T}{I}))$, so $\sigma(U(\sfrac{T}{I}))=\ker(\tau)$. Therefore, if every ideal class in $\ker(\tau)$ contains an ideal class representative $J$ which is an integral ideal of $R$ comaximal to $I$, then the sequence $(\iota,\phi,\psi,\sigma,\tau)$ is exact.

    On the other hand, assume that the sequence $(\iota,\phi,\psi,\sigma,\tau)$ is exact. In particular, this means that $\sigma(U(\sfrac{T}{I}))=\ker(\tau)$. Then every ideal class in $\ker(\tau)$ can be represented as $\sigma(t+I)=[tT\cap R]$ for some $t+I\in U(\sfrac{T}{I})$. Clearly $tT\cap R$ is an integral ideal of $R$; since $t+I\in U(\sfrac{T}{I})$, then $tT$ is comaximal to $I$ in $T$. Then by Lemma \ref{still invertible}, $tT\cap R$ is comaximal to $I$ in $R$. Thus, every ideal class in $\ker(\tau)$ contains an ideal class representative $J$ which is an integral ideal of $R$ comaximal to $I$.

    Finally, assume that, in addition to the previous condition, every ideal class in $\Cl(T)$ contains an ideal class representative $J$ which is an integral ideal of $T$ comaximal to $I$. In order to show that the sequence $(\iota,\phi,\psi,\sigma,\tau,\pi)$ is exact, we only need to show that $\tau$ is surjective. To do so, let $[A]\in\Cl(T)$. By the assumption, we can without loss of generality assume that $A$ is an integral ideal of $T$ comaximal to $I$. Then by Lemma \ref{still invertible}, $R\cap A$ is an invertible integral ideal of $R$ comaximal to $I$; that is, $[R\cap A]\in \Cl(R)$. Moreover, by Lemma \ref{going up or going down}, $\tau([R\cap A])=[(R\cap A)T]=[A]$. Then $\tau$ is surjective, so the sequence $(\iota,\phi,\psi,\sigma,\tau,\pi)$ is exact.
\end{proof}

This exact sequence gives us exactly what we need to conclude that the alternate characterization of locally associated orders given in \cite[Theorem 3.4]{subringrelations} holds more generally.

\begin{theorem}
    Let $T$ be an integral domain, $R$ a subring with identity, and $I=(R:T)$. If the sequence $(\iota,\phi,\psi,\sigma,\tau,\pi)$ given in Theorem \ref{exact sequence} is exact (in particular, if every ideal class in $\ker(\tau)$ contains an ideal class representative which is an integral ideal of $R$ comaximal to $I$ and every ideal class in $\Cl(T)$ contains an ideal class representative which is an integral ideal of $T$ comaximal to $I$), then the following are equivalent:
    \begin{enumerate}
        \item $R$ is a locally associated subring of $T$;
        \item $\Cl(R)\cong \Cl(T)$ via the isomorphism $\tau$.
    \end{enumerate}
\end{theorem}

\begin{proof}
    Since the sequence from Theorem \ref{exact sequence} is exact, it follows that $\tau$ is an isomorphism if and only if $\sigma$ is trivial, i.e., $\ker(\sigma)=U(\sfrac{T}{I}).$ Thus, condition (2) is equivalent to $\psi$ being surjective. In this case, for every $t+I\in U(\sfrac{T}{I})$, there is some $u\in U(T)$ and $r+I\in U(\sfrac{R}{I})$ such that $t+I=ur+I$. Then for some $\beta\in I$, $t=ur+\beta=(r+u^{-1}\beta)u.$ Then by characterization (1) in Definition \ref{locally associated}, $R$ is a locally associated subring of $T$. The converse follows by simply reversing the previous argument.
\end{proof}

As an immediate consequence of \cite[Theorem 5.3]{conrad} and the fact that any order in a number field has finite class group and ideal norms, we get the following.

\begin{corollary}
    Let $R$ and $T$ be orders in a number field with $R\subseteq T\subseteq \Rbar$. Then the following hold:
    \begin{enumerate}
        \item The sequence $(\iota,\phi,\psi,\sigma,\tau,\pi)$ given in Theorem \ref{exact sequence} is exact;
        \item The class number of $R$ is given by
        $$\abs{\Cl(R)}=\abs{\Cl(T)}\frac{\abs{U(\sfrac{T}{I})}}{\abs{U(\sfrac{R}{I})}\cdot\abs{\quot{U(T)}{U(R)}}}.$$
    \end{enumerate}
    Moreover, the following are equivalent:
    \begin{enumerate}
        \item $R$ is a locally associated subring of $T$;
        \item $\abs{\quot{U(T)}{U(R)}}=\frac{\abs{U(\sfrac{T}{I})}}{\abs{U(\sfrac{R}{I})}};$
        \item $\abs{\Cl(R)}=\abs{\Cl(T)}$;
        \item $\Cl(R)\cong \Cl(T)$.
    \end{enumerate}
\end{corollary}

\section{Polynomial and Power Series Extensions with Changing Coefficients}

We now turn our attention to the rings of primary interest to this paper: rings of polynomials and formal power series in which the coefficient ring is allowed to change with each power of $x$. As we have previously discussed, these rings are often useful when searching for examples of unexpected behavior. We begin by defining a compact notation for such a ring.

\begin{definition}
    Let $\Gamma$ be an additively closed subset of $\bN_0$, $\{R_\alpha\}_{\alpha\in\Gamma}$ a collection of commutative rings indexed by $\Gamma$ such that $R_a\subseteq R_b$ for all $a,b\in \Gamma$ with $b-a\in\Gamma$. Also let $R=\bigcup_{\alpha\in \Gamma} R_i$. We define the \term{ring of polynomials in $x$ over $\{R_\alpha\}_{\alpha\in\Gamma}$} to be the subring of $R[x]$ given by
    $$R_{\Gamma}[x;\Gamma]:=\left\{\sum_{i=1}^n r_ix^{\alpha_i}\middle|n\in\bN, \alpha_i\in \Gamma, r_i\in R_{\alpha_i}\right\}.$$
    Similarly, we define the \term{ring of formal power series in $x$ over $\{R_\alpha\}_{\alpha\in\Gamma}$} to be the subring of $R[[x]]$ given by
    $$R_{\Gamma}[[x;\Gamma]]:=\left\{\sum_{\alpha\in\Gamma}r_\alpha x^\alpha\middle|r_\alpha\in R_\alpha\right\}.$$
    That is, $R_{\Gamma}[x;\Gamma]$ is the set of polynomials in $R[x]$ in which the coefficient of $x^\alpha$ is restricted to $R_\alpha$ for all $\alpha\in \Gamma$ (and similarly for $R_{\Gamma}[[x;\Gamma]]$). When $\Gamma=\bNo$, we will use the simpler shorthand $R_{\bNo}[x]$ and $R_\bNo[[x]]$.
\end{definition}

We take this opportunity to record a few notes about this notation. First, the $\Gamma$ which appears in the subscript of $R$ is meant to indicate the rings from which each coefficient may be drawn; the $\Gamma$ following the semicolon is meant to indicate the exponents of $x$ which are allowed. Second, note that we require $R_a\subseteq R_b$ whenever $b-a\in \Gamma$ to ensure that $R_\Gamma[x;\Gamma]$ and $R_\Gamma[[x;\Gamma]]$ are closed under multiplication. When $\Gamma=\bNo$, this condition simplifies to the requirement that $R_i\subseteq R_{i+1}$ for each $i\in\bNo$. 

For the majority of this paper, we will be focused on the case when $\Gamma=\bNo$. That said, we define these rings more generally to have access to the notation when needed. It is also worth noting that a regular ring of polynomials $R[x]$ (respectively, $R[[x]]$) is simply $R_{\bNo}[x]$ (respectively, $R_\bNo[[x]]$) with $R_i=R$ for all $i\in\bNo$. Thus, any results we find for these more general rings will apply to $R[x]$ and $R[[x]]$ as well.



Our primary goal in this paper is to determine when a subring of the form $R_{\bN_0}[x]$ (respectively, $R_{\bN_0}[[x]]$) is an associated or locally associated subring of a ring of the form $T_{\bN_0}[x]$ (respectively, $T_{\bN_0}[[x]]$). As such, it will be important to characterize the conductor ideal $(R_{\bN_0}[x]:T_{\bN_0}[x])$ (respectively, $(R_{\bN_0}[[x]]:T_{\bN_0}[[x]])$) and the units in these rings. We begin with the conductor ideal.

\begin{proposition}
    \label{conductor ideal}
    Let $\{R_i\}_{i=0}^\infty$ and $\{T_i\}_{i=0}^\infty$ be sequences of commutative rings such that $R_i\subseteq R_{i+1}$, $T_i\subseteq T_{i+1}$, and $R_i\subseteq T_i$ for all $i\in\bN_0$. For each $i\in\bN_0$, define $J_i:=\bigcap_{k=0}^\infty (R_{i+k}:T_k)$. Then $(R_{\bN_0}[x]:T_{\bN_0}[x])=J_{\bN_0}[x]$ and $(R_{\bN_0}[[x]]:T_{\bN_0}[[x]])=J_{\bN_0}[[x]]$.
\end{proposition}

\begin{proof}
    We will show that $(R_{\bN_0}[x]:T_{\bN_0}[x])=J_{\bN_0}[x]$; the proof for power series is nearly identical.

    First, assume that $f=a_0+a_1x+\dots+a_nx^n\in (R_{\bN_0}[x]:T_{\bN_0}[x])$; that is, $f\in R_{\bN_0}[x]$, and if $g\in T_{\bN_0}[x]$, then $fg\in R_{\bN_0}[x]$. Then in particular, for any $k\in\bN_0$ and $t_k\in T_k$, $t_kx^kf=(a_0t_k)x^k+(a_1t_k)x^{k+1}+\dots+(a_nt_k)x^{k+n}\in R_{\bN_0}[x]$. Since $a_0\in R_0\subseteq R_1\subseteq\dots$ must conduct any $t_k\in T_k$ into $R_k$, $a_0\in (R_k:T_k)$ for every $k\in\bN_0$. Then $a_0\in J_0$. More generally, since $a_i\in R_i\subseteq R_{i+1}\subseteq\dots$ must conduct any $t_k\in T_k$ into $R_{i+k}$, $a_i\in (R_{i+k}:T_k)$ for every $k\in\bN_0$. Then $a_i\in J_i$ for every $i\in\bN_0$. Thus, $(R_{\bN_0}[x]:T_{\bN_0}[x])\subseteq J_{\bN_0}[x]$.

    Now let $f=a_0+a_1x+\dots+a_nx^n\in J_{\bNo}[x]$; that is, $a_i\in J_i$ for each $i\in\bNo$ (in particular, $a_i\in J_i\subseteq (R_i:T_0)\subseteq R_i$, so $f\in R_{\bNo}[x]$). Also let $g(x)=t_0+t_1x+\dots+t_mx^m\in T_{\bNo}[x]$. Then $(fg)(x)=\sum_{k=0}^{n+m}(a_0t_k+a_1t_{k-1}+\dots+a_kt_0)x^k$ (with $t_i=0$ and $a_j=0$ for $i>m$ and $j>n$). Since each $a_i\in J_i$, then for each $k\in\bN_0$, $a_i\in (R_{i+k}:T_k)$. Then in the above summation, $a_it_{k-i}\in R_k$, so $fg\in R_\bNo[x]$. Thus, $J_{\bNo}[x]\subseteq (R_{\bNo}[x]:T_{\bNo}[x])$, so we conclude that $(R_{\bNo}[x]:T_{\bNo}[x])= J_{\bNo}[x]$.
\end{proof}

\begin{remark}
    With the notation as above, note that $J_i$ must be an ideal of $R_i$ for each $i\in\bN_0$ and thus must be a subring of $T_i$. However, $J_i$ is not necessarily an ideal of $T_i$. Nevertheless, since $R_{\bNo}[x]$ is a subring of $T_{\bNo}[x]$ and $J_{\bNo}[x]=(R_{\bNo}[x]:T_{\bNo}[x])$, it is not difficult to see that $J_{\bNo}[x]$ is an ideal of both $R_{\bNo}[x]$ and $T_{\bNo}[x]$ (similarly for power series).

    When we consider the quotient rings $\quot{R_{\bNo}[x]}{J_{\bNo}[x]}$ and $\quot{T_{\bNo}[x]}{J_{\bNo}[x]}$, we will need to keep in mind two things. First, note that $J_i$ is not necessarily an ideal of $T_i$; thus, we cannot consider term-by-term ring quotients $\quot{T_i}{J_i}$. However, if we instead treat $J_i$ as an additive subgroup of $T_i$, we can treat $\quot{T_{\bNo}[x]}{J_{\bNo}[x]}$ as the set $(\quot{T_\bNo}{J_\bNo})[x]$, where $\quot{T_{i}}{J_{i}}$ is understood to be the set of additive cosets of $T_i$ modulo $J_i$ (similarly for power series). The second thing to note is that even though the $\quot{R_i}{J_i}$'s are valid quotient rings, we may still not be able to treat $\quot{R_i}{J_i}$ as a subring of $\quot{R_{i+1}}{J_{i+1}}$ (even by an embedding). Thus, in both cases we will need to be careful when multiplying cosets.

    As a final note, we observe that in the constant term, $J_0$ must be an ideal of $T_0$. To see this, let $\alpha\in J_0$ and $t\in T_0$. Then for any $t_k\in T_k$, note that $tt_k\in T_k$, so $(\alpha t)t_k=\alpha(tt_k)\in R_k$. Thus, $\alpha t\in (R_k:T_k)$ for every $k\in\bN_0$, so $\alpha t\in J_0$.
\end{remark}

\begin{example}
    \label{conductor ideal example}
    Let $T_0=R_0=R_1=R_2=\bZ$ and $R_i=T_j=\bQ$ for $i> 2$ and $j>0$. Then $J_0=J_1=\{0\}$ (since $(R_1:T_1)=(R_2:T_1)=\{0\}$), $J_2=\bZ$ (since $(R_2:T_0)=\bZ$ and $(R_{k+2}:T_k)=\bQ$ for every $k\in\bN$), and $J_i=\bQ$ for every $i>2$ (since $(R_{k+i}:T_k)=\bQ$ for every $k\in\bN_0$---note here that it is important that $(\bQ:\bZ)=\{\alpha\in\bQ|\alpha\bZ\subseteq\bQ\}=\bQ$ rather than $\{\alpha\in\bZ|\alpha\bZ\subseteq\bQ\}=\bZ$). In this case, note that $J_2=\bZ$ is not an ideal of $T_2=\bQ$. However, $J_{\bNo}[x]=x^2\bZ+x^3\bQ[x]$ is still an ideal of $T_{\bNo}[x]=\bZ+x\bQ[x]$; in fact, $J_{\bNo}[x]$ is the principal ideal generated by $x^2$. Also note that $\quot{R_0}{J_0}=\quot{R_1}{J_1}\cong \bZ$ cannot be embedded as a subring into $\quot{R_2}{J_2}\cong \{0\}$. Then for instance, when multiplying the nonzero cosets $(x+J_{\bN_0}[x])(x+J_\bNo[x])$, we would get the zero coset $J_{\bNo}[x]$.
\end{example}

Typically when dealing with usual rings of polynomials or formal power series, the units are easy to characterize using standard results: the units in a polynomial ring are those whose constant coefficient is a unit and whose remaining terms are nilpotent; and the units in a ring of formal power series are those whose constant coefficient is a unit. As it turns out, the groups of units in rings of polynomials or power series over a sequence of rings have a very similar structure. Despite the proofs following by largely the same argument as the more standard results, we include them here for completeness.

\begin{proposition}
    \label{unit construction}
    Let $R_i$, $T_i$, and $J_i$ be as in Proposition \ref{conductor ideal} for $i\in\bN_0$. Also define $R:=\bigcup_{i=0}^\infty R_i$, $T:=\bigcup_{i=0}^\infty T_i$, and $J:=\bigcup_{i=0}^\infty J_i$. Then:
    \begin{enumerate}
        \item The units in $T_{\bNo}[x]$ are exactly the polynomials whose constant term is a unit and whose remaining terms are nilpotent in $T$ (equivalently, nilpotent in $T_i$). That is,
        $$U(T_{\bNo}[x])=U(T_0)+\Nil(T_{\bN})[x;\bN].$$
        The same will hold for $R_{\bNo}[x]$.
        \item The units in $T_{\bNo}[[x]]$ are exactly the power series whose constant term is a unit. That is,
        $$U(T_{\bNo}[[x]])=U(T_0)+T_{\bN}[[x;\bN]].$$
        The same will hold for $R_{\bNo}[[x]]$.
        \item The units in $\quot{
        T_{\bNo}[x]}{J_{\bNo}[x]}$ are exactly the cosets represented by polynomials in $T_{\bNo}[x]$ whose constant terms are units modulo $J_0$ and whose remaining terms are nilpotent modulo $T_i\cap J$ (equivalently, nilpotent as elements of $T$ modulo $J$). That is, in a slight abuse of notation,
        $$U(\quot{
        T_{\bNo}[x]}{J_{\bNo}[x]})=U(\quot{T_0}{J_0})+\Nil(\quot{T_\bN}{T_{\bN}\cap J})[x;\bN].$$
        The same will hold for $\quot{R_{\bNo}[x]}{J_{\bNo}[x]}$.
        \item The units in $\quot{
        T_{\bNo}[[x]]}{J_{\bNo}[[x]]}$ are exactly the cosets represented by power series in $T_{\bNo}[[x]]$ whose constant terms are units modulo $J_0$. That is, in a slight abuse of notation,
        $$U(\quot{
        T_{\bNo}[[x]]}{J_{\bNo}[[x]]})=U(\quot{T_0}{J_0})+(\quot{T_{\bN}}{J_{\bN}})[[x;\bN]].$$
        The same will hold for $\quot{R_{\bNo}[[x]]}{J_{\bNo}[[x]]}$.
    \end{enumerate}
\end{proposition}

\begin{proof}[Proof of (1)]
    First, suppose that $f=a_0+a_1x+\dots+a_nx^n\in U(T_{\bNo}[x])$. Then there exists $g=b_0+b_1x+\dots+b_nx^n\in U(T_{\bNo}[x])$ such that $fg=1$. The constant term in this product is $a_0b_0=1$, so $a_0\in U(T_0)$. Furthermore, since $f$ and $g$ also both lie in $T[x]$, then $f\in U(T[x])$. Then $f=a_0+a_1x+\dots+a_nx^n$, with $a_1,\dots,a_n\in \Nil(T)$. Any element of $T_i$ which is nilpotent in $T$ is also nilpotent in $T_i$, so for $1\leq i\leq n$, $a_i\in \Nil(T_i)$. Thus, $U(T_\bNo[x])\subseteq U(T_0)+\Nil(T_\bN)[x;\bN]$. Moreover, an element $f=a_0+a_1x+\dots+a_nx^n\in U(T_0)+\Nil(T_\bN)[x;\bN]$ is clearly the sum of a unit $a_0\in U(T_{\bNo}[x])$ plus the nilpotent elements $a_ix^i$, so $f\in U(T_\bNo[x])$. Thus, $U(T_{\bNo}[x])=U(T_0)+\Nil(T_{\bN})[x;\bN].$ The same logic shows that $U(R_{\bNo}[x])=U(R_0)+\Nil(R_{\bN})[x;\bN].$
\end{proof}

\begin{proof}[Proof of (2)]
    The fact that the constant term of any unit in $T_{\bNo}[[x]]$ must be a unit in $T_0$ follows exactly as in the proof of (1). Then all that needs to be shown for this part is that any power series whose constant term is a unit must itself be a unit. To show this, let $f=u+a_1x+a_2x^2+\cdots\in U(T_0)+T_\bN[x;\bN]$. To show that $f$ is invertible, we will find $b_i\in T_i$ for $i\in\bN$ such that $fg=1$, with $g=u^{-1}+b_1x+b_2x^2+\cdots$. As in the standard proof characterizing the units in a ring of formal power series, we set
    $$b_k=-u^{-1}(u^{-1}a_k+b_1a_{k-1}+\cdots+b_{k-1}a_1)$$
    for each $k\in\bN$. Note that each element on the right-hand side of this expression is in $T_k$, so $b_k\in T_k$; thus, $g\in T_\bNo[[x]]$. Then in the expansion of the product $fg$, we find that the constant term is $uu^{-1}=1$ and each higher-degree coefficient is 0. Then $f\in U(T_{\bNo}[[x]])$ with $f^{-1}=g$, so $U(T_{\bNo}[[x]])=U(T_0)+T_{\bN}[[x;\bN]]$. The same logic shows that $U(R_{\bNo}[[x]])=U(R_0)+R_{\bN}[[x;\bN]].$
\end{proof}

\begin{proof}[Proof of (3)]
    Again, the fact that for $f=a_0+a_1x+\dots+a_nx^n\in T_{\bNo}[x]$ to be invertible modulo $J_{\bNo}[x]$, $a_0$ must be a unit modulo $J_0$ (i.e., $a_0+J_0\in U(\quot{T_0}{J_0})$) is obvious and follows in the same manner as before. Furthermore, if $f$ is a unit modulo $J_{\bNo}[x]$, there must exist $g\in T_{\bNo}[x]$ and $\beta\in J_{\bNo}[x]$ such that $fg+\beta=1$. Since $f,g\in T[x]$ and $\beta\in J[x]$, this tells us that $f$ is also an element of $T[x]$ which is invertible modulo $J[x]$. Thus, each $a_i$ for $1\leq i\leq n$ must be nilpotent modulo $J$. This means that some power of $a_i$ lies in $T_i\cap J$, so $a_i\in \sqrt{T_i\cap J}$. Therefore, $U(\quot{T_\bNo[x]}{J_\bNo[x]})\subseteq U(\quot{T_0}{J_0})+\Nil(\quot{T_\bN}{T_\bN\cap J})[x;\bN]$. Moreover, any $f+J_\bNo[x]=a_0+a_1x+\dots+a_nx^n+J_\bNo[x]\in U(\quot{T_0}{J_0})+\Nil(\quot{T_\bN}{T_\bN\cap J})[x;\bN]$ is the sum of the unit $a_0+J_\bNo[x]$ and the nilpotent elements $a_ix^i+J_\bNo[x]$ for $1\leq i\leq n$. Then $U(\quot{T_{\bNo}[x]}{J_{\bNo}[x]})=U(\quot{T_0}{J_0})+\Nil(\quot{T_\bN}{T_{\bN}\cap J})[x;\bN].$ The same logic shows that $U(\quot{R_{\bNo}[x]}{J_{\bNo}[x]})=U(\quot{R_0}{J_0})+\Nil(\quot{R_\bN}{R_{\bN}\cap J})[x;\bN].$
\end{proof}

\begin{proof}[Proof of (4)]
    The fact that the constant term of any unit power series must be a unit modulo $J_0$ follows as usual. We must simply show that any power series in $T_{\bNo}[[x]]$ whose constant term is invertible modulo $J_0$ is in fact a unit modulo $J_{\bNo}[[x]]$. To do so, let $f=a_0+a_1x+a_2x^2+\cdots\in T_\bNo[[x]]$, with $a_0+J_0\in U(\quot{T_0}{J_0})$. We must construct a power series $g=b_0+b_1x+b_2x^2+\cdots$ such that $fg=1+h$, with $h\in J_{\bNo}[[x]]$. First, we select $b_0\in T_0$ such that $a_0b_0=1+\beta_0$ for some $\beta_0\in J_0$. Now assume for some $k\in\bN$ that we have selected $b_i\in T_i$ for $0\leq i<k$. By expanding the product $fg$ and solving for the term $a_0b_k$ in the coefficient of $x^k$, we see that we now need to select $b_k\in T_k$ and $\beta_k\in J_k$ such that $a_0b_k+\beta_k=-(a_kb_0+a_{k-1}{b_1}+\cdots+a_1b_{k-1})$. Since the right-hand side of this equation is an element of $T_k$, it will suffice to show that $a_0T_k+J_k=T_k$. We already know that $a_0T_0+J_0=T_0$. Then consider the set $J_0T_k$. Note that for any $i\in\bN_0$, $T_kT_i= T_{\max\{k,i\}}\subseteq T_{k+i}$. Then $(J_0T_k)T_i = J_0(T_kT_i)\subseteq J_0T_{k+i}\subseteq R_{k+i}$, so $J_0T_k\subseteq J_k$. Therefore, $T_k=T_0T_k=(a_0T_0+J_0)T_k\subseteq a_0T_k+J_0T_k\subseteq a_0T_k+J_k\subseteq T_k$. Then $a_0T_k+J_k=T_k$, so we can indeed select appropriate elements $b_k\in T_k$ and $\beta_k\in J_k$ satisfying the necessary relation. Then $f$ is invertible modulo $J_\bNo[[x]]$, so the units in $\quot{T_{\bNo}[[x]]}{J_{\bNo}[[x]]}$ are exactly the cosets represented by power series whose constant terms are units modulo $J_0$. That is, $U(\quot{T_{\bNo}[[x]]}{J_{\bNo}[[x]]})=U(\quot{T_0}{J_0})+(\quot{T_{\bN}}{J_{\bN}})[[x;\bN]].$ The same logic shows that $U(\quot{R_{\bNo}[[x]]}{J_{\bNo}[[x]]})=U(\quot{R_0}{J_0})+(\quot{R_{\bN}}{J_{\bN}})[[x;\bN]].$
\end{proof}

\section{(Locally) Associated Subrings in Polynomial Extensions}

With these preliminaries handled, we turn our attention to the major questions of this paper. In this section, we will characterize when $R_\bNo[x]$ is an associated or locally associated subring of $T_\bNo[x]$ for sequences $\{R_i\}_{i=0}^\infty$ and $\{T_i\}_{i=0}^\infty$. The results in this section will be easiest to state and work with when the $T_i$'s are reduced rings (i.e., when the $T_i$'s have no nonzero nilpotent elements --- note that this will automatically hold if the $T_i$'s are integral domains). However, we state them as generally as we are able. For the convenience of the reader, we will also include two corollaries to each major result. The first will restate the result as it applies to standard polynomial rings $R[x]\subseteq T[x]$; the second will narrow further to the motivating case when $R$ is an order in a number field and $T=\Rbar$ is the full ring of integers.

\begin{theorem}
    \label{associated polynomials}
    Let $\{R_i\}_{i=0}^\infty$ and $\{T_i\}_{i=0}^\infty$ be increasing sequences of commutative rings with identity such that $R_i\subseteq T_i$ for each $i\in\bN_0$. Let $S=R_0\cap U(T_0)$ (a multiplicative subset of $T_0$), and assume that for each $i\in\bN$, $\Nil(T_i)=S^{-1}\Nil(R_i)$ (in particular, this holds when each $T_i$ is reduced). Then $R_{\bNo}[x]$ is an associated subring of $T_{\bNo}[x]$ if and only if the following conditions both hold:
    \begin{enumerate}
        \item $R_0$ is an associated subring of $T_0$;
        \item $T_i=S^{-1}R_i$ for each $i\in\bN$.
    \end{enumerate}
\end{theorem}

\begin{proof}
    First, assume that $R_\bNo[x]$ is an associated subring of $T_\bNo[x]$. Then for any $t_0\in T_0$, there must exist $u=u_0+b_1x+\dots+b_nx^n\in U(T_\bNo[x])$ such that $t_0u\in R_{\bNo}[x]$. By Proposition \ref{unit construction}, this means that $u_0\in U(T_0)$ and $b_i\in \Nil(T_i)$ for each $1\leq i\leq n$. In particular, $t_0u_0\in R_0$, so $R_0$ must be an associated subring of $T_0$.

    Now for some $i\in\bN$, let $t_i\in T_i$ and consider $1+t_ix^i\in T_\bNo[x]$. Since $R_\bNo[x]$ is an associated subring of $T_\bNo[x]$, there must exist some $u=u_0+b_1x+\dots+b_nx^n\in U(T_\bNo[x])$ such that $(1+t_ix^i)u\in R_{\bNo}[x]$. The constant term of this product is $u_0$, so $u_0\in S=R_0\cap U(T_0)$. The coefficient of $x^i$ must be $u_0t_i+b_i=r_i\in R_i$. By Proposition \ref{unit construction}, we know that $b_i\in \Nil(T_i)=S^{-1}\Nil(R_i)$, so $b_i=s^{-1}r$ for some $s\in S$ and $r\in R_i$. Thus,
    $$t_i=\frac{sr_i-r}{u_0s}\in S^{-1}R_i,$$
    so $T_i\subseteq S^{-1}R_i$. The reverse inclusion is obvious, so $T_i=S^{-1}R_i$ for each $i\in\bN$.

    For the converse, assume that $R_0$ is an associated subring of $T_0$ and $T_i=S^{-1}R_i$ for each $i\in\bN$, and let $f=t_0+t_1x+\dots+t_nx^n\in T_\bNo[x]$. Since $R_0$ is an associated subring of $T_0$, there exists some $u_0\in U(T_0)$ such that $t_0u_0\in R_0$. Furthermore, since $T_i=S^{-1}R_i$ for each $i\in\bN$, there must exist $r_i\in R_i$ and $s_i\in S$ such that $t_i=s_i^{-1}r_i$ for each $1\leq i\leq n$. Then $(u_0s_1\cdots s_n)f\in R_\bNo[x]$ with $u_0s_1\cdots s_n\in U(T_0)\subseteq U(T_\bNo[x])$. Therefore, $R_\bNo[x]$ is an associated subring of $T_\bNo[x]$.
\end{proof}

\begin{corollary}
    Let $T$ be a commutative ring with identity, $R$ a subring of $T$, and $S=R\cap U(T)$. If $\Nil(T)=S^{-1}\Nil(R)$ (in particular, if $T$ is reduced), then $R[x]$ is an associated subring of $T[x]$ if and only if $T$ is a localization of $R$ (specifically, $T=S^{-1}R$).
\end{corollary}

\begin{corollary}
    Let $R$ be an order in a number field. Then $R[x]$ is an associated subring of $\Rbar[x]$ if and only if $R=\Rbar$. That is, $R[x]$ is not an associated subring of $\Rbar[x]$ for any non-maximal order $R$.
\end{corollary}

\begin{example}
\label{polynomial not associated}
    $\bZ+x\bZ+x^2\bQ[x]$ is not an associated subring of $\bZ+x\bQ[x]$. To see this, note that even though $\bQ$ is a localization of $\bZ$, it is not the localization by $S=\bZ\cap U(\bZ)=\{\pm 1\}$. Thus, condition (2) in Theorem \ref{associated polynomials} is not satisfied. For instance, there is no $u\in U(\bZ+x\bQ[x])=\{\pm 1\}$ which multiplies $\frac{1}{2}x\in \bZ+x\bQ[x]$ into $\bZ+x\bZ+x^2\bQ[x]$.
\end{example}

\begin{example}
\label{polynomial associated}
    Let $\{T_i\}_{i=0}^\infty$ be an increasing sequence of number fields, and for each $i\in\bNo$, let $R_i$ be any order in $T_i$. Since $T_0$ is a field, then $R_0$ is automatically an associated subring of $T_0$. Furthermore, for any $\alpha_i\in T_i$, $\alpha_i=\frac{r_i}{n_i}$ for some $r_i\in R_i$ and $n_i\in \bN$. Since $\bN\subseteq R_0\cap U(T_0)$, then each $T_i$ is the localization of $R_i$ by $S=R_0\cap U(T_0)$. Thus, $R_\bNo[x]$ is an associated subring of $T_\bNo[x]$. 
\end{example}

We now provide a similar characterization of when $R_\bNo[x]$ is a locally associated subring of $T_\bNo[x]$.

\begin{theorem}
    \label{la polynomial}
    Let $\{R_i\}_{i=0}^\infty$ and $\{T_i\}_{i=0}^\infty$ be increasing sequences of commutative rings with identity such that $R_i\subseteq T_i$ for each $i\in\bNo$, $J_i=\bigcap_{k=0}^\infty (R_{i+k}:T_k)$ for each $i\in\bNo$, and $J=\bigcup_{i=0}^\infty J_i$. Also assume that for each $i\in\bN$, $\Nil(T_i)\subseteq R_i$ (in particular, this holds when each $T_i$ is reduced). Then $R_\bNo[x]$ is a locally associated subring of $T_\bNo[x]$ if and only if the following conditions both hold:
    \begin{enumerate}
        \item for every element $t_0\in T_0$ comaximal to $J_0$ (i.e., $t_0T_0+J_0=T_0$), there exists a unit $u\in U(T_0)$ such that $t_0u$ is an element of $R_0$ comaximal to $J_0$ (i.e., $t_0uR_0+J_0=R_0$);
        \item for every $k\in\bN$, $\sqrt{T_k\cap J}\subseteq R_k$.
    \end{enumerate}
\end{theorem}

\begin{proof}
    First, assume that $R_\bNo[x]$ is a locally associated subring of $T_\bNo[x]$. Then for any element $t_0\in T_0$ comaximal to $J_0$, Proposition \ref{unit construction} tells us that $t_0$ is also comaximal to $J_\bNo[x]$ as an element of $T_\bNo[x]$. Since $R_\bNo[x]$ is locally associated, there must exist some $u=u_0+b_1x+\dots+b_nx^n\in U(T_\bNo[x])$ such that $t_0u+J_\bNo[x]\in U(\quot{R_\bNo[x]}{J_\bNo[x]})$. In particular, $t_0u_0$ must be an element of $R_0$ comaximal to $J_0$ by Proposition \ref{unit construction}, satisfying Condition (1).

    Now let $t_k\in \sqrt{T_k\cap J}$; that is, $t_k\in T_k$ such that $t_k^n\in J$ for some $n\in\bN$. Then by Proposition \ref{unit construction}, the polynomial $1+t_kx^k$ is a unit modulo $J_0$. Since $R_\bNo[x]$ is locally associated, there must again be some $u=u_0+b_1x+\dots+b_nx^n\in U(T_\bNo[x])$ such that $(1+t_kx^k)u+J_\bNo[x]\in U(\quot{R_\bNo[x]}{J_\bNo[x]})$. In particular, the constant term $u_0$ is an element of $R_0$ comaximal to $J_0$. Then there exist some $r_0\in R_0$ and $\beta_0\in J_0$ such that $u_0r_0+\beta_0=1$. Furthermore, the coefficient of $x^k$, $t_ku_0+b_k$, is an element of $R_k$ which is nilpotent modulo $J$. Thus:
    $$t_k=t_k(u_0r_0+\beta_0)=(t_ku_0+b_k)r_0-b_kr_0+t_k\beta_0.$$
    Since $t_ku_0+b_k$ and $r_0$ are elements of $R_k$, so is their product. Since $\beta_0\in J_0$, then $t_k\beta_0\in R_k$. Finally, since $b_k\in \Nil(T_k)\subseteq R_k$, $b_kr_0\in R_k$. Thus, $t_k\in R_k$, so $\sqrt{T_k\cap J}\subseteq R_k$.

    Now assume that Conditions (1) and (2) hold, and let $f=t_0+t_1x+\dots+t_nx^n$ be an element of $T_\bNo[x]$ comaximal to $J_\bNo[x]$. In particular, Proposition \ref{unit construction} tells us that $t_0+J_0\in U(\quot{T_0}{J_0})$ and $t_i\in \sqrt{T_i\cap J}$ for each $i\in\bN$. By Condition (1), there exists some $u\in U(T_0)$ such that $t_0u$ is an element of $R_0$ comaximal to $J_0$. Then $fu=(t_0u+t_1ux+\dots+t_nux^n)$, with each $t_ku\in \sqrt{T_k\cap J}\subseteq R_k$. Then $t_ku\in \sqrt{R_k\cap J}$; by Proposition \ref{unit construction}, $fu$ is an element of $R_\bNo[x]$ comaximal to $J_\bNo[x]$. Thus, $R_\bNo[x]$ is a locally associated subring of $T_\bNo[x]$.
\end{proof}

\begin{corollary}
    Let $T$ be a commutative ring with identity, $R$ a subring of $T$ with identity, and $J=(R:T)$. If $\Nil(T)\subseteq R$ (in particular, if $T$ is reduced), then $R[x]$ is a locally associated subring of $T[x]$ if and only if $R$ is a locally associated subring of $T$ such that $J$ is a radical ideal of $T$.
\end{corollary}

\begin{corollary}
    Let $R$ be an order in a number field and $J=(R:\Rbar)$. Then $R[x]$ is a locally associated subring of $T[x]$ if and only if $R$ is a locally associated subring of $T$ and $J$ is a radical ideal of $T$ (i.e., $J$ factors into a product of distinct prime ideals of $T$).
\end{corollary}

\begin{example}
\label{polynomial not la}
    Let $\{p_i\}_{i=1}^\infty$ be an enumeration of the prime numbers, and for each $i\in\bN$, define $\alpha_i=\frac{1+\sqrt{p_i}}{2}$ if $p_i\equiv 1\modulo{4}$ and $\alpha_i=\sqrt{p_i}$ otherwise. Define $T_0=\bZ[i]$ and $R_0=\bZ[2i]$, and for each $i\in\bN$, define $T_i=T_{i-1}[\alpha_i]$ and $R_i=T_{i-1}[2\alpha_i]$. Then $T=\bigcup_{i=0}^\infty T_i$ is the smallest domain containing all quadratic rings of integers. For each $i\in\bN_0$, note that $(R_i:T_i)=2T_i$ and $(R_j:T_i)=R_j$ for any $j>i$ (since $T_i\subseteq R_j$). Then $J_0=2T_0=(R_0:T_0)$ and $J_i=R_i$ for all $i\in\bN$. By \cite[Theorem 6.4]{subringrelations}, $R_0$ is a locally associated subring of $T_0$, so Condition (1) in Theorem \ref{la polynomial} holds. However, note that $J=\bigcup_{i=0}^\infty J_i=\bigcup_{i=0}^\infty R_i=T$. Then for any $k\in\bN$, $\sqrt{T_k\cap J}=\sqrt{T_k\cap T}=\sqrt{T_k}=T_k\nsubseteq R_k$. Thus, $R_\bNo[x]$ is not a locally associated subring of $T_\bNo[x]$.
\end{example}

\begin{example}
\label{polynomial la}
    Similarly to the previous example, let $T_0=\bZ[i]$ and $R_0=\bZ[2i]$. Letting $K_0=\bQ[i]$, we inductively define for each $i\in\bN$ a field extension $K_i$ of $K_{i-1}$ such that $2$ remains inert in $T_i$, the ring of algebraic integers in $K_i$ (that is, $2T_i$ is a prime ideal of $T_i$). Finally, we let $R_i=\bZ+2T_i$ for all $i\in\bN$. Note that for each $i,k\in\bN_0$, $(R_{i+k}:T_k)=2T_{i+k}$, a prime ideal of $T_{i+k}$. Then $J_i=2T_i$ for each $i\in\bN_0$ and $J=2T$. As before, we note that since $R_0$ is a locally associated subring of $T_0$ and $J_0=2T_0=(R_0:T_0)$, condition (1) from Theorem \ref{la polynomial} is satisfied. Furthermore, since $T_k\cap J=2T_k$ is a prime ideal for each $k\in\bN$, then $\sqrt{T_k\cap J}=2T_k\subseteq R_k$, satisfying condition (2) as well. Then $R_\bNo[x]$ is a locally associated subring of $T_\bNo[x]$.
\end{example}

\section{(Locally) Associated Subrings in Power Series Extensions}

We now consider the analogous questions for power series extensions: when is $R_\bNo[[x]]$ an associated or locally associated subring of $T_\bNo[[x]]$? Unfortunately, a general characterization of when $R_\bNo[[x]]$ is an associated subring of $T_\bNo[[x]]$ remains elusive. That said, we are able to provide a set of sufficient conditions as well as a set of necessary conditions for $R[[x]]$ to be associated in $T[[x]]$. At the intersection of these conditions will be a full characterization of when $R[[x]]$ is associated in $T[[x]]$ for a certain class of domains. To start, we provide lemmata which generalize some of the results from \cite[Section 4]{subringrelations}.

\begin{lemma}
    \label{conductor of R+J}
    Let $R\subseteq T$ be commutative rings, $I=(R:T)$, and $J$ an ideal of $T$ containing $I$. If $R$ is an ideal-preserving subring of $T$ and $I$ is the intersection of finitely many maximal ideals of $T$, then $R+J$ is a subring of $T$ with $J=(R+J:T)$.
\end{lemma}

\begin{proof}
    Let $M_1,\dots, M_n$ be the maximal ideals of $T$ such that $I=\bigcap_{i=1}^n M_i$. Without loss of generality, we assume that this list is ordered such that for some $1\leq k\leq n$, $J\subseteq M_i$ for each $1\leq i\leq k$ and $J\nsubseteq M_j$ for each $k<j\leq n$. Since $J$ is an ideal containing $I$, then note that $\quot{J}{I}$ is an ideal of $\quot{T}{I}\cong \quot{T}{M_1}\times\dots\times\quot{T}{M_n}$, a direct product of finitely many fields. Then letting $\pi_i:\quot{T}{I}\to\quot{T}{M_i}$ be the natural projection, $\pi_i(\quot{J}{I})$ must be an ideal of the field $\quot{T}{M_i}$ for each $1\leq i\leq n$. Since the only ideals of a field are the zero ideal and the entire field, then for each $1\leq i\leq n$, $\pi_i(\quot{J}{I})$ is either the zero ideal of $\quot{T}{M_i}$ (in which case $J\subseteq M_i$) or the entire field $\quot{T}{M_i}$ (in which case $J\nsubseteq M_i$). Thus, $\quot{J}{I}\cong \pi_1(\quot{J}{I})\times\dots\times \pi_n(\quot{J}{I})\cong \quot{T}{M_{k+1}}\times\dots\times\quot{T}{M_n}$. Then by the Third Isomorphism Theorem,
    $$\quot{T}{J}\cong\quot{(\sfrac{T}{I})}{(\sfrac{T}{J})}\cong \quot{T}{M_1}\times\dots\times \quot{T}{M_k},$$ so $J=\bigcap_{i=1}^k M_i$.

    Let $A=(R+J:T)$. Since $J\subseteq R+J$, then $J\subseteq A$; notably, this means that $A$ is also an intersection of finitely many maximal ideals (specifically, a subset of the maximal ideals containing $J$). Assume toward a contradiction that $J\subsetneq A$. Then there is some maximal ideal containing $J$ which does not contain $A$; assume without loss of generality that this maximal ideal is $M_1$. Since $R$ is an ideal-preserving subring of $T$, we can therefore select $\alpha\in R\cap A\backslash M_1$ and $\beta\in R\cap (\bigcap_{i=2}^{n}M_i)\backslash M_1$. Note that since $\alpha\in A=(R+J:T)$, then $\alpha T\subseteq R+J$. Thus, for any $t\in T$, $\alpha t=r+j$ for some $r\in R$ and $j\in J$. Then $r\beta\in R$ and $j\beta\in \bigcap_{i=1}^nM_i=I$, so $(\alpha\beta)t=r\beta+j\beta\in R$. Thus $\alpha\beta T\subseteq R$, so $\alpha\beta\in I$. However, note that $\alpha\beta\notin M_1$, a contradiction. Then $J=(R+J:T)$.
\end{proof}

\begin{lemma}
    \label{intersection of intermediate subrings}
    Let $R\subseteq T$ be commutative rings and $I=(R:T)$. Also assume that $R$ is an ideal-preserving subring of $T$, that $I=\bigcap_{i=1}^n M_i$ for some maximal ideals $M_i$ of $T$, and that $R\cap M_1,\dots,R\cap M_n$ are pairwise comaximal ideals of $R$. If $J_1$ and $J_2$ are comaximal ideals of $T$ containing $I$, then $(R+J_1)\cap (R+J_2)=R+J_1\cap J_2$.
\end{lemma}

\begin{proof}
    Since the $R\cap M_i$'s are pairwise comaximal as ideals of $R$, it follows immediately from the Chinese Remainder Theorem that
    $$\quot{R}{I}\cong \quot{R}{R\cap M_1}\times\dots\times \quot{R}{R\cap M_n}.$$
    The Second Isomorphism Theorem then tells us that
    \begin{equation}\quot{R}{I}\cong\quot{R}{R\cap M_1}\times\dots\times \quot{R}{R\cap M_n}\cong \quot{(R+M_1)}{M_1}\times\dots\times \quot{(R+M_n)}{M_n}.\end{equation}
    Then let $J_1$ and $J_2$ be comaximal ideals of $T$ containing $I$. By Lemma \ref{conductor of R+J}, $R+J_1\cap J_2$ is an ideal-preserving subring whose conductor ideal $(R+J_1\cap J_2:T)=J_1\cap J_2$ is an intersection of finitely many maximal ideals of $T$. Moreover, since the $M_i$'s remain comaximal when restricted to $R$, they will certainly remain comaximal when restricted to $R+J_1\cap J_2$. Then it will suffice to show that when $J_1\cap J_2=I$, $(R+J_1)\cap (R+J_2)=R$.
    
    As shown in the proof of the previous lemma, $J_1$ and $J_2$ are both intersections of finitely many of the maximal ideals containing $I$. Moreover, since $J_1$ and $J_2$ are comaximal, none of the maximal ideals containing $J_1$ also contain $J_2$ (and vice versa). Then after possibly rearranging the $M_i$'s, there exists $1\leq r\leq n$ such that $J_1=\bigcap_{i=1}^r M_i$ and $J_2=\bigcap_{i=r+1}^n M_i$. Applying the isomorphisms in (1) to each $\quot{(R+J_i)}{J_i}$ (noting that $R+J_i$, by the previous lemma, has conductor ideal $J_i$), we have:
    $$\quot{R}{I}\cong \quot{(R+J_1)}{J_1}\times \quot{(R+J_2)}{J_2}$$
    under the isomorphism $\phi:\quot{R}{I}\to \quot{(R+J_1)}{J_1}\times \quot{(R+J_2)}{J_2}$ defined by $\phi(r+I)=(r+J_1,r+J_2)$.

    Now note that $R\subseteq (R+J_1)\cap (R+J_2)$ trivially. For the reverse inclusion, let $\alpha\in (R+J_1)\cap (R+J_2)$. Then there is some $r\in R$ such that $\phi(r+I)=(r+J_1,r+J_2)=(\alpha+J_1,\alpha+J_2)$. Then $\alpha-r\in J_1\cap J_2=I$, so $\alpha=r+(\alpha-r)\in R$. Therefore, $R=(R+J_1)\cap (R+J_2)$.
\end{proof}

\begin{lemma}
    \label{la in towers}
    Let $R\subseteq T$ be commutative rings and $I=(R:T)$. Also suppose that $R$ is an ideal-preserving and locally associated subring of $T$ and that $I$ is the intersection of finitely many maximal ideals of $T$. Then for any ideal $J$ of $T$ containing $I$, $R+J$ is a locally associated subring of $T$.
\end{lemma}

\begin{proof}
    Let $I=\bigcap_{i=1}^n M_i$. Without loss of generality, assume we have ordered the $M_i$'s as in the proof of Lemma \ref{conductor of R+J} so that for some $1\leq k\leq n$, $J=\bigcap_{i=1}^k M_i$. Let $t\in T$ be comaximal to $J=(R+J:T)$ (i.e., $t\notin M_i$ for $1\leq i\leq k$). By the Chinese Remainder Theorem, we can now select $\beta\in T$ such that $\beta\equiv t\modulo{M_i}$ for $1\leq i\leq k$ and $\beta\equiv 1\modulo{M_j}$ for $k<j\leq n$. Note that $\beta\notin M_i$ for each $1\leq i\leq n$ and is thus comaximal to $I$. Since $R$ is a locally associated subring of $T$, there exists some $u\in U(T)$ such that $u\beta$ is an element of $R$ comaximal to $I$. Note that $ut=u(\beta+(t-\beta))=u\beta+u(t-\beta)$, with $t-\beta\in \bigcap_{i=1}^k M_i=J$. Then since $u\beta\in R$ and $u(t-\beta)\in J$, $ut\in R+J$.
    
    Since $u\beta$ is comaximal to $I$ as an element of $R$, then we can select $r\in R$ and $\gamma\in I$ such that $u\beta r+\gamma=1$. Therefore:
    $$1=u\beta r+\gamma=u(t+(\beta-t))r+\gamma=utr+[(\beta-t)ur+\gamma].$$
    Since $\beta\equiv t\modulo{M_i}$ for each $1\leq i\leq k$, then $\beta-t\in \bigcap_{i=1}^k M_i=J$. Then $r\in R\subseteq R+J$ and $(\beta-t)ur+\gamma\in J$, so $ut$ is an element of $R+J$ which is comaximal to $J=(R+J:T)$ as an element of $R+J$. Therefore, $R+J$ is a locally associated subring of $T$.
\end{proof}

\begin{theorem}
    \label{associated ps sufficient}
    Let $R\subseteq T$ be commutative rings and $I=(R:T)$. Also suppose that the following conditions hold:
    \begin{enumerate}
        \item $R$ is both an associated subring and a locally associated subring of $T$ (by \cite[Theorem 2.11]{subringrelations}, if $U(\sfrac{R}{I})=\quot{R}{I}\cap U(\sfrac{T}{I})$, then this condition is equivalent to $R$ just being associated);
        \item $I$ is the intersection of finitely many maximal ideals $M_1,\dots,M_n$ of $T$;
        \item The ideals $R\cap M_1,\dots,R\cap M_n$ are pairwise comaximal as ideals of $R$.
    \end{enumerate}
    Then $R[[x]]$ is an associated subring of $T[[x]]$.
\end{theorem}

\begin{proof}
    First, assume that $I$ is itself a maximal ideal of $T$ and consider $f=t_0+t_1x+t_2x^2+\cdots\in T[[x]]$. To show that $R[[x]]$ is an associated subring of $T[[x]]$, we will show that there exist $u=u_0+b_1x+b_2x^2+\dots\in U(T[[x]])$ and $r=r_0+r_1x+r_2x^2+\cdots \in R[[x]]$ such that $f=ru$. If $f\in I[[x]]$, then we can select $u=1$ and $r=f$. Then assume that $f\notin I[[x]]$; for now, also assume that $t_0\notin I$. Since $I$ is maximal in $T$, this means that $t_0T+I=T$.

    Since $R$ is an associated subring of $T$, there must exist $u_0\in U(T)$ and $r_0\in R$ such that $t_0=r_0u_0$. Since $u_0\in U(T)$, note that $r_0T+I=t_0T+I=T$. Now assume that for some $n\in\bN$, we have selected $r_0,r_1,\dots,r_{n-1}$ and $u_0,b_1,\dots,b_{n-1}$ such that for each $k<n$, $$t_k=u_0r_k+b_1r_{k-1}+\dots+b_kr_0.$$
    We simply need to show that we can now select $r_n\in R$ and $b_n\in T$ such that $$t_n=u_0r_n+b_1r_{n-1}+\dots+b_nr_0;$$
    in other words, such that
    $$u_0r_n+b_nr_0=t_n-(b_1r_{n-1}+\dots+b_{n-1}r_1).$$
    Since we have previously selected all $b_k$'s and $r_k$'s for $k<n$, note that the right-hand side of this equation is simply a fixed element of $T$. Then since $u_0R+r_0T\supseteq u_0I+r_0T\supseteq r_0T+I=T$, we can select $r_n\in R$ and $b_n\in T$ which satisfy the above identity. By induction, this selection of $r_n$ and $b_n$ is possible for every $n\in\bN$, and thus we can construct $u=u_0+b_1x+b_2x^2+\dots\in U(T[[x]])$ and $r=r_0+r_1x+r_2x^2+\dots\in R[[x]]$ such that $f=ru$.

    Still assuming that $I$ is maximal and that $f\notin I[[x]]$, we now drop the assumption that $t_0\notin I$. Since $f\notin I[[x]]$, then we can let $n\in\bN$ be minimal such that $t_n\notin I$ (i.e., $t_n\notin I$, but $t_k\in I$ for all $k<n$). Then $f=(t_0+t_1x+\dots+t_{n-1}x^{n-1})+x^n(t_n+t_{n+1}x+\dots)=a+bx^n$, with $a\in I[x]$ and $b$ having constant term $t_n\notin I$. By the previous case, there must exist some $r\in R[[x]]$ and $u\in U(T[[x]])$ such that $b=ru$. Then $f=(au^{-1}+rx^n)u$, where $au^{-1}+rx^n\in R[[x]]$. Then $R$ is an associated subring of $T$.

    We now drop the assumption that $I$ is maximal and let $f=t_0+t_1x+t_2x^2+\dots\in T[[x]]$ be arbitrary. By Lemma \ref{conductor of R+J}, note that each $R+M_i$ is a subring of $T$ whose conductor ideal $(R+M_i:T)=M_i$ is a maximal ideal of $T$; moreover, Lemma \ref{la in towers} tells us that $R+M_i$ is a locally associated subring of $T$. Then by the previous case, for each $1\leq i\leq n$, there exists $g_i=u_i+b_{i1}x+b_{i2}x^2+\dots \in U(T[[x]])$ such that $fg_i\in (R+M_i)[[x]]$. Now by the Chinese Remainder Theorem, there must exist $b_0\in T$ such that $b_0\equiv u_i\modulo{M_i}$ for each $1\leq i\leq n$. Moreover, for each $j\in \bN$, there must exist $b_j\in T$ such that $b_j\equiv b_{ij}\modulo{M_i}$ for each $1\leq i\leq n$. Let $g=b_0+b_1x+b_2x^2+\dots \in T[[x]]$ and note that, by construction, $fg\equiv fg_i\modulo{M_i[[x]]}$ for each $1\leq i\leq n$. Then $fg\in \bigcap_{i=1}^n(R+M_i)[[x]]$, which by Lemma \ref{intersection of intermediate subrings} gives that $fg\in R[[x]]$.

    Finally, note that since $b_0$ is congruent to a unit modulo each $M_i$, $b_0$ is comaximal to $I$. Since $R$ is a locally associated subring of $T$, this means that there exists some $r\in R$ comaximal to $I$ (as an element of $R$) and $\beta\in I$ such that $b_0r+\beta=u_0\in U(T)$. Then let $u=gr+\beta=u+rb_1x+rb_2x^2+\dots\in U(T[[x]])$. Therefore, $fu=f(gr+\beta)=(fg)r+f\beta\in R[[x]]$, so $R[[x]]$ is an associated subring of $T[[x]]$.
\end{proof}

This gives a set of conditions which are sufficient to conclude that $R[[x]]$ is an associated subring of $T[[x]]$. Following a short lemma, we now provide a related (but not equivalent) set of necessary conditions.

\begin{lemma}
    \label{associate not in R}
    Let $T$ be a commutative ring with unity, $R$ an associated subring of $T$, and $I=(R:T)$. If $r\in R$ satisfies $r\cdot U(T)\subseteq R$, then $r\in I$. Equivalently, if $r\notin I$, then there exists $u\in U(T)$ such that $ur\notin R$.
\end{lemma}

\begin{proof}
    Let $r\in R$ such that $r\cdot U(T)\subseteq R$, and let $t\in T$. Since $R$ is an associated subring of $T$, there exists $u\in U(T)$ such that $tu\in R$. Then $rt=(ru^{-1})(tu)\in R$, so $r\in (R:T)=I$.
\end{proof}

\begin{theorem}
    \label{associated ps necessary}
    Let $R\subsetneq T$ be commutative rings with unity such that $U(\sfrac{R}{I})= \quot{R}{I}\cap U(\sfrac{T}{I})$ (in particular, this holds when $T$ is integral over $R$). If $R[[x]]$ is an associated subring of $T[[x]]$, then $R$ is an associated subring of $T$ and $I=(R:T)$ is a radical ideal of $T$.
\end{theorem}

\begin{proof}
    Assume that $R[[x]]$ is an associated subring of $T[[x]]$. In particular, this means that for any $t\in T$, there exists $u=u_0+b_1x+b_2x^2+\cdots\in U(T[[x]])$ such that $tu\in R[[x]]$. Then the constant term $tu_0\in R$ with $u_0\in U(T)$, so $R$ is an associated subring of $T$.

    Now assume toward a contradiction that $I$ is not a radical ideal of $T$. Then there is some $\alpha\in T\backslash I$ for which $\alpha^2\in I$. By Lemma \ref{associate not in R}, there exists some $v\in U(T)$ such that $v\alpha\notin R$. Then consider $v\alpha+x\in T[[x]]\backslash R[[x]]$. Since $R[[x]]$ is an associated subring of $T[[x]]$, there must exist $u=u_0+b_1x+b_2x^2+\dots\in U(T[[x]])$ (i.e., $u_0\in U(T)$ and $b_i\in T$ for $i\in\bN$) such that $(v\alpha+x)u\in R[[x]]$. In particular, this means that $u_0v\alpha \in R$ and $r:=v\alpha b_1+u_0\in R$. Note that since $\alpha\in\sqrt{I}$, then $v\alpha b_1$ lies in every maximal ideal containing $I$. Since $u_0\in U(T)$ and $r=v\alpha b_1+u_0$, this means that $r$ lies in none of the maximal ideals containing $I$; thus, $r+I\in \quot{R}{I}\cap U(\sfrac{T}{I})=U(\sfrac{R}{I})$. Then there exists $s\in R$ such that $sr\equiv 1\modulo{I}$.

    Now note the following:
    $$u_0v\alpha=(r-v\alpha b_1)(v\alpha)=r(v\alpha)-\alpha^2 v^2b_1\in R.$$
    Since $\alpha^2\in I$, then $\alpha^2 vb_1\in R$. Thus, $\beta:=r(v\alpha)\in R$. Then $v\alpha\equiv s\beta\modulo{I}$, and thus $v\alpha\in R$, contradicting the choice of $v$. Then $I$ must be a radical ideal of $T$.
\end{proof}

Combining these two theorems provides a full characterization of when $R[[x]]$ is an associated subring of $T[[x]]$ in the case that $R$ is an order in a number field and $T=\Rbar$, the full ring of integers.

\begin{corollary}
    Let $R$ be a non-maximal order in a number field and $I=(R:\Rbar)$. Then $R[[x]]$ is an associated subring of $\Rbar[[x]]$ if and only if $R$ is an associated order and $I$ is a radical ideal of $\Rbar$ (i.e., $I$ factors into a products of distinct prime ideals of $T$).
\end{corollary}

\begin{proof}
    If $R[[x]]$ is an associated subring of $\Rbar[[x]]$, then Theorem \ref{associated ps necessary} immediately tells us that $R$ is an associated order with radical conductor ideal $I$. We must show the converse.

    Assume that $R$ is an associated order whose conductor ideal $I$ is a radical ideal of $\Rbar$. Since $R$ is an associated order, it is both ideal-preserving and locally associated. Thus, condition (1) of Theorem \ref{associated ps sufficient} is satisfied. Furthermore, since $I$ is radical in $\Rbar$, $I=P_1\cdots P_k=\bigcap_{i=1}^kP_i$ for some distinct primes $P_1,\dots,P_k$ of $\Rbar$. Since $\Rbar$ is Dedekind, each $P_i$ must be maximal; thus, condition (2) of Theorem \ref{associated ps sufficient} is satisfied. Finally, since $R$ is an ideal-preserving order and $\dim(R)=1$, the ideals $R\cap P_1,\dots,R\cap P_k$ are pairwise distinct prime (and thus maximal) ideals of $R$. Then the $R\cap P_i$'s are pairwise comaximal as ideals of $R$, so condition (3) of Theorem \ref{associated ps sufficient} is satisfied. Then by Theorem \ref{associated ps sufficient}, $R[[x]]$ is an associated subring of $T[[x]]$.
\end{proof}

\begin{example}
    \label{ps not associated}
    Let $T=\bigcup_{i=0}^\infty T_i$ be the domain constructed in Example \ref{polynomial not la} and $R=\bZ+2T$. Then $T$ is integral over $R$, but $I:=(R:T)=2T$ is not a radical ideal of $T$ (in particular, $\sqrt{2}\in T\backslash I$, but $2=(\sqrt{2})^2\in I$). By Theorem \ref{associated ps necessary}, $R[[x]]$ is not an associated subring of $T[[x]]$.
\end{example}

\begin{example}
    \label{ps associated}
    Let $p_1,\dots,p_k$ be (finitely many) inert primes in the quadratic number ring $\bZ[\sqrt{2}]$; that is, $P_i:=p_i\bZ[\sqrt{2}]$ is a prime ideal for each $1\leq i\leq k$. Let $T$ be the localization of $\bZ[\sqrt{2}]$ by the multiplicatively closed set $S=\bigcap_{i=1}^kP_i^C$ so that every element of $\bZ[\sqrt{2}]$ which does not lie in any of the $P_i$'s becomes a unit. We will let $M_i:=P_iT$ be the maximal ideal of $T$ lying over $P_i$ for each $1\leq i\leq k$, $I:=M_1\dots M_k=(p_1\cdots p_k)T$, and $R:=\bZ+I$. 
    
    First, note that $I=(R:T)$. To see this, let $\alpha\in R$ such that $\alpha\sqrt{2}\in R$. Then $\alpha=n+\beta$ for some $\beta\in I$ and $n\sqrt{2}+\beta\sqrt{2}\in R$, so $n\sqrt{2}\in R$. Then there exist $m\in \bZ$, $t\in \bZ[\sqrt{2}]$, and $s\in S$ such that $n\sqrt{2}=m+\frac{p_1\cdots p_kt}{s}$. Then $s(-m+n\sqrt{2})\in p_1\cdots p_k\bZ[\sqrt{2}]$. Since $s\notin P_i$ for each $1\leq i\leq k$, then $-m+n\sqrt{2}\in P_i$, and thus $n\in p_1\cdots p_k\bZ$. Then $\alpha\in I$, so $(R:T)\subseteq I$; the reverse inclusion is obvious. Now note that $R$ is a subring of $T$ whose conductor ideal $I=(R:T)$ is the intersection of the finitely many maximal ideals $M_1,\dots,M_n$ of $T$; moreover, since the $M_i$'s lie over distinct prime integers, $R\cap M_1,\cdots,R\cap M_k$ are comaximal ideals of $R$.

    Finally, let $\alpha\in T$. If $\alpha$ is comaximal to $I$, then $\alpha\in U(T)$; thus, there is some $u=\alpha^{-1}\in U(T)$ such that $\alpha u=1$ is an element of $R$ comaximal to $I$. Then $R$ is a locally associated subring of $T$. If $\alpha$ is not comaximal to $I$, assume without loss of generality that for some $1\leq j\leq k$, $\alpha\in M_i$ for $1\leq i\leq j$ and $\alpha\notin M_i$ for $j<i\leq k$. Since each $M_i$ is maximal, there exists $t_i\in T$ for each $j<i\leq k$ such that $t_i\alpha\equiv 1\modulo{M_i}$. By the Chinese Remainder Theorem, we may select $u\in T$ such that $u\equiv 1\modulo{M_i}$ for $1\leq i\leq j$ and $u\equiv t_i\modulo{M_i}$ for $j<i\leq k$. Since $u$ does not lie in any of the $M_i$'s, it follows that $u\in U(T)$. Moreover, $u\alpha\in \bigcap_{i=1}^k(\bZ+M_i)=R$. Then $R$ is also an associated subring of $T$. By Theorem \ref{associated ps sufficient}, $R[[x]]$ is an associated subring of $T[[x]]$.
\end{example}

Before moving on to the question of locally associated power series extensions, we take a moment to apply these results to a question about half-factorial power series over orders. Notably, the approach we use to tackle this question is very similar to that from the proof of Theorem \ref{associated ps necessary}. We will also make frequent use of the following lemmata, which are refinements of \cite[Lemma 6.5]{radicalconductor} and Lemma \ref{associate not in R}, respectively.

\begin{lemma}
    \label{non-HFD construction}
    Let $R$ be an order in a number field with conductor ideal $I=(R:\Rbar)$. Suppose that there exists a power series $f=g(h+k)\in \Irr(R[[x]])$, with $g\in R[[x]]$ a nonunit, $h\in \Rbar[[x]]$ such that $h^n\in R[[x]]$ for all $n\geq 2$, and $k\in R[[x]]$ such that $h+k$ is a nonunit. Then $R[[x]]$ is not an HFD.
\end{lemma}

\begin{proof}
    First, note that since $R$ is an order in a number field, there must exist some $m\in \bZ\cap I$ (in particular, $m=|\quot{\Rbar}{I}|$ will always work). Then expanding $(h+k)^m$ yields:
    $$(h+k)^{m}=\sum_{i=2}^{m}\binom{m}{i}h^ik^{m-i}+mhk^{m-1}+k^{m}.$$
    Since $k\in R[[x]]$, then $k^{m}\in R[[x]]$ as well. Since $m\in I$, then $m(hk^{m-1})\in I[[x]]\subseteq R[[x]]$. Finally, since $\binom{m}{i}\in\bZ\subseteq R$, $h^i\in R[[x]]$ for each $i\geq 2$, and $k\in R[[x]]$, then each term in the summation is also an element of $R[[x]]$. Thus, $(h+k)^m\in R[[x]]$. Then in $R[[x]]$, $g^m$ can be factored into at least $m$ irreducible elements and $(h+k)^m$ can be factored into at least one irreducible (since $h+k$ is a nonunit). Then $f^m=g^m(h+k)^m$ can simultaneously be written as a product of $m$ irreducibles in $R[[x]]$ and strictly more than $m$ irreducibles. Thus, $R[[x]]$ is not an HFD.
\end{proof}

\begin{lemma}
    \label{associate in intermediate orders}
    Let $R$ be an associated order in a number field with conductor ideal $I=(R:\Rbar)$. Also assume that $I=J_1J_2$, with $J_1$ and $J_2$ relatively prime ideals of $\Rbar$. If $\alpha\notin J_1$, then there exists some $u\in U(T)$ such that $u\alpha\notin R+J_1$ and $u\alpha\in R+J_2$.
\end{lemma}

\begin{proof}
    Since $R$ is an associated order, it must be both ideal-preserving and locally associated. Then \cite[Theorem 4.5]{subringrelations} tells us that $(R+J_1:\Rbar)=J_1$. Moreover, it will suffice to show the result under the assumption that $\alpha\in R$, since every element of $\Rbar$ has an associate in $R$ and association is transitive.

    Let $\alpha\in R\backslash J_1$. By Lemma \ref{associate not in R}, there exists some $u_1\in U(\Rbar)$ such that $u_1\alpha\notin R+J_1$. By the Chinese Remainder Theorem, we may select $t\in \Rbar$ such that $t\equiv u_1\modulo{J_1}$ and $t\equiv 1\modulo{J_2}$. Since $t$ is comaximal to $I$ and $R$ is locally associated, there exist $u\in U(\Rbar)$, $\beta\in I$, and $r\in R$ comaximal to $I$ such that $tr+\beta=u$. Then $u\alpha=(tr+\beta)\alpha=(t\alpha)r+\beta\alpha\equiv u_1\alpha r\modulo{J_1}$. Since $r\in R$ is invertible modulo $J_1$ and $u_1\alpha\notin R+J_1$, it follows that $u\alpha\notin R+J_1$. Furthermore, $u\alpha=(t\alpha)r+\beta\alpha\equiv \alpha r\modulo{J_2}$, so $u\alpha\in R+J_2$.
\end{proof}

\begin{theorem}
    \label{HFD power series over order}
    Let $R$ be an order in a number field with conductor ideal $I=(R:\Rbar)$. If the following conditions all hold, then $R[[x]]$ is an HFD:
    \begin{enumerate}
        \item $\Rbar$ is an HFD;
        \item $R$ is an associated order;
        \item $I$ is a radical ideal of $\Rbar$.
    \end{enumerate}
    If $R[[x]]$ is an HFD, then the following conditions hold:
    \begin{enumerate}
        \item $\Rbar$ is an HFD;
        \item $R$ is an associated order;
        \item $I$ factors into prime ideals $P_1^{a_1}\cdots P_k^{a_k}$, with $a_i=1$ if $P_i$ is principal in $\Rbar$ and $a_i\leq 2$ if $P_i$ is non-principal in $\Rbar$.
    \end{enumerate}
\end{theorem}

\begin{proof}
    The first part of this result follows immediately from \cite[Theorems 3.8, 5.9]{radicalconductor}. For the second part, assume that $R[[x]]$ is an HFD and let $\alpha\in R$ be a nonzero, nonunit. Since $R$ is atomic, we may factor $\alpha$ into a product of irreducible elements. Suppose we have done so in two ways: $\alpha=\pi_1\cdots\pi_m=\tau_1\cdots\tau_n.$ Since the $\pi_i$'s and $\tau_j$'s are irreducible elements of $R$, they must also be irreducible in $R[[x]]$. Then $\alpha\in R[[x]]$ can simultaneously be factored into a product of $m$ and $n$ irreducibles in the HFD $R[[x]]$, so $m=n$. Then $R$ is also an HFD. By \cite[Theorem 1.1]{rago}, it follows that $\Rbar$ is an HFD and $R$ is an associated order.

    Now write $I=P_1^{k_1}\cdots P_n^{k_n}$, and assume toward a contradiction that, without loss of generality, either $P_1$ is principal with $k_1>1$ or $P_1$ is non-principal with $k_1>2$. Let $J_1=P_1^{k_1}$ and $J_2=P_2^{k_2}\cdots P_n^{k_n}$, and note by \cite[Theorem 4.6]{subringrelations} that $R=(R+J_1)\cap (R+J_2)$. If $P_1$ is a principal ideal of $\Rbar$ and $k_1>1$, we let $P_1=\pi\Rbar$; since $R$ is an associated order, we may assume without loss of generality that $\pi\in R$. Since $\pi^{k_1-1}\notin J_1$, Lemma \ref{associate in intermediate orders} tells us that there must exist some $v\in U(T)$ such that $v\pi^{k_1-1}\in (R+J_2)\backslash(R+J_1)$. Then consider the nonzero, nonunit power series $f:=\pi(v\pi^{k_1-1}+x)=v\pi^{k_1}+\pi x\in R[[x]]$, and suppose that $f$ reduces in $R[[x]]$. Then there exist nonunit power series $g=a_0+a_1x+\cdots$ and $h=b_0+b_1x+\cdots$ in $R[[x]]$ such that $f=gh$.

    Considering the constant term of this product, we have $a_0b_0=v\pi^{k_1}$. Since $\pi^{k_1}\Rbar=P^{k_1}$, then this element exhibits unique factorization in $\Rbar$. Then for some $1\leq j<k_1$ and $u\in U(\Rbar)$, $a_0=uv\pi^j\in R$ and $b_0=u^{-1}\pi^{k_1-j}$; without loss of generality, we assume that $k_1-j\geq j$. Now considering the constant term of $f=gh$, we have $$\pi=a_0b_1+a_1b_0=uv\pi^jb_1+u^{-1}\pi^{k_1-j}a_1=\pi^j(uvb_1+u^{-1}\pi^{k_1-2j}a_1).$$
    If $j>1$, then $\pi^2|\pi$, a contradiction. Thus, $j=1$ and $1=uvb_1+u^{-1}\pi^{k_1-2}a_1$. If $k_1>2$, this means that $uvb_1\in (R+P_1)\backslash P_1$. If $k_1=2$, then since $f=gh\in P_1[[x]]$, a prime ideal of $\Rbar[[x]]$, either $g$ or $h$ must lie in $P_1[[x]]$. Since, in this case, $g$ and $h$ are interchangeable, we may assume without loss of generality that $g\in P_1[[x]]$ so that $a_1\in P_1$; again, this tells us that $uvb_1\in (R+P_1)\backslash P_1$. Then $b_1\in R$ is invertible modulo $P_1$, so $uv\in R+P_1$ as well. Then $b_0=u^{-1}\pi^{k_1-1}\in R\cap P^{k_1-1}$ and $uv\in R+P_1$, so $v\pi^{k_1-1}=(uv)b_0\in (R+P)(R\cap P^{k-1})\subseteq R+J_1$, a contradiction. Then $f=\pi(v\pi^{k_1-1}+x)\in \Irr(R[[x]])$. Since $(v\pi^{k_1-1})^m\in J_1\cap (R+J_2)\subseteq R$ for all $m\geq 2$, Lemma \ref{non-HFD construction} tells us that $R[[x]]$ is not an HFD.

    Now assume that $P_1$ is a non-principal prime ideal and $k_1>2$ is even. Then we let $\beta\in R$ be a generator of $P_1^2$ and $r=\frac{k_1}{2}\geq 2$. Since $\beta^{r-1}\notin J_1$, Lemma \ref{associate in intermediate orders} tells us that there exists some $v\in U(\Rbar)$ such that $v\beta^{r-1}\in (R+J_2)\backslash(R+J_1)$. In a similar manner to the previous case, we consider the nonzero, nonunit power series $f=\beta(v\beta^{r-1}+x)=v\beta^{r}+\beta x\in R[[x]].$
    As before, suppose that $f=gh$ reduces in $R[[x]]$.
    
    Note that $(\beta^{r})=P_1^{k_1}$, so $a_0b_0=v\beta^{r}$ exhibits unique factorization into irreducibles in $\Rbar$. Then for some $u\in U(\Rbar)$ and $1\leq j< r$, $a_0=uv\beta^j$ and $b_0=u^{-1}\beta^{r-j}$; as before, we assume without loss of generality that $r-j\geq j$. Considering the linear term of $f=gh$, we have
    $$\beta=a_0b_1+a_1b_0=uv\beta^jb_1+u^{-1}\beta^{r-j}a_1.$$
    If $j>1$, then $\beta^2\mid \beta$, a contradiction; then $j=1$ and $1=uvb_1+u^{-1}\beta^{r-2}a_1$. If $r>2$, this means that $uvb_1\in (R+P_1)\backslash P_1$. If $r=2$, we conclude as in the previous case that we may assume without loss of generality that $a_1\in P_1$, again telling us that $uvb_1\in (R+P_1)\backslash P_1.$ In any case, $b_1\notin P_1$. Now considering the quadratic coefficient of $f=gh$, we have
    $$0=a_0b_2+a_1b_1+a_2b_0=uv\beta b_2+a_1b_1+u^{-1}\beta^{r-1}a_2.$$
    In particular, note that $0$, $\beta$, and $\beta^{r}$ are all elements of $P_1^2$, so $a_1b_1\in P_1^2$ as well. Since $b_1\notin P_1$, this means that $a_1\in P_1^2$. Then $uvb_1=1-u^{-1}\beta^{r}a_1\in R+P_1^2$; since $b_1\in R$ is invertible modulo $P_1$, this tells us that $uv\in R+P_1^2$ as well. Then $v\beta^{r-1}=(uv)(u^{-1}\beta^{r-1})\in (R+P_1^2)(R\cap P_1^{k_1-2})\subseteq R+J_1$, a contradiction. Then $f=\beta(v\beta^{r-1}+x)\in \Irr(R[[x]])$. Since $(v\beta^{r-1})^m\in J_1\cap (R+J_2)\subseteq R$ for all $m\geq 2$, Lemma \ref{non-HFD construction} once again tells us that $R[[x]]$ is not an HFD.

    Finally, assume that $P_1$ is a non-principal prime and $k_1> 2$ is odd. We let $\beta\in R$ be a generator of $P_1^2$, and for some non-principal prime $Q\neq P_1$, we let $\alpha\in R$ be a generator of $P_1Q$. Let $r=\frac{k_1-1}{2}\geq 1$; since $\beta^r\notin J_1$, then Lemma \ref{associate in intermediate orders} tells us that there exists $v\in U(\Rbar)$ such that $v\beta^r\in(R+J_2)\backslash (R+J_1)$. We now consider the nonzero, nonunit power series $f=\alpha(v\beta^r+x)=v\alpha\beta^r+\alpha x\in R[[x]]$
    and suppose as before that $f=gh$ reduces in $R[[x]]$.
    
    Since $(\alpha\beta^r)=P^kQ$, we again note that $\alpha\beta^r$ admits unique factorization in $\Rbar$. Then since $a_0b_0=v\alpha\beta^r$, there must exist $u\in U(\Rbar)$ and $0\leq j<r$ such that, without loss of generality, $a_0=uv\alpha\beta^j$ and $b_0=u^{-1}\beta^{r-j}$. Considering the linear term of $f=gh$, we have $$\alpha=a_0b_1+a_1b_0=uv\alpha\beta^jb_1+u^{-1}\beta^{r-j}a_1.$$
    Since $\alpha\in P_1\backslash P_1^2$ and $\beta^{r-j}\in P_1^2$, then $uv\alpha\beta^jb_1\notin P_1^2$. Notably, this means that $j=0$ and $b_1\notin P_1$. Thus, $\alpha(1-uvb_1)=u^{-1}\beta^ra_1\in P_1^{k_1-1}$, so $1-uvb_1\in P^{k_1-2}$ with $k_1>2$. As before, since $uvb_1\in R+P_1$ and $b_1\in R\backslash P_1$, it follows that $uv\in R+P_1$. Then $v\beta^r=(uv)(u^{-1}\beta^r)\in(R+P_1)(R\cap P_1^{k-1})\subseteq R+J_1$, a contradiction. Then $f=\alpha(v\beta^r+x)\in \Irr(R[[x]])$, so Lemma \ref{non-HFD construction} again tells us that $R[[x]]$ is not an HFD.
\end{proof}

\begin{remark}
    Note that Theorem \ref{HFD power series over order} does not completely characterize when the ring of formal power series $R[[x]]$ over an order $R$ is an HFD; however, it comes quite close. The only unsettled cases come when $R$ is an associated order in the HFD $\Rbar$ whose conductor ideal is exactly divisible by $P^2$ for some non-principal prime ideal $P$. An example of such a ring $R[[x]]$ which is not half-factorial was provided in \cite[Theorem 6.8]{radicalconductor}. Whether such a ring may be half-factorial is still an open question.
\end{remark}

Finally, we consider when $R_\bNo[[x]]$ is a locally associated subring of $T_\bNo[[x]]$. As it turns out, this case exhibits the most ready inheritance of all.

\begin{theorem}
    \label{la power series}
    Let $\{R_i\}_{i=0}^\infty$ and $\{T_i\}_{i=0}^\infty$ be increasing sequences of commutative rings with identity such that $R_i\subseteq T_i$ for each $i\in\bNo$, and let $J_i=\bigcap_{i=0}^\infty(R_{i+k}:T_i)$ for each $i\in\bNo$. Then $R_\bNo[[x]]$ is a locally associated subring of $T_\bNo[[x]]$ if and only if, for every $t_0\in T_0$ comaximal to $J_0$, there exists $u\in U(T_0)$ such that $t_0u$ is an element of $R_0$ comaximal to $J_0$.
\end{theorem}

\begin{proof}
    First, assume that $R_\bNo[[x]]$ is a locally associated subring of $T_\bNo[[x]]$, and let $t_0$ be an element of $T_0$ comaximal to $J_0$. By Proposition \ref{unit construction}, $t_0$ is also an element of $T_\bNo[[x]]$ comaximal to $J_\bNo[[x]]$, so there must exist $u=u_0+b_1x+b_2x^2+\cdots\in U(T_\bNo[[x]])$ such that $t_0u$ is an element of $R_\bNo[[x]]$ comaximal to $J_\bNo[[x]]$. In particular, this means that the constant term $t_0u_0$ must be an element of $R_0$ comaximal to $J_0$. Then for every $t_0\in T_0$ comaximal to $J_0$, there exists $u_0\in U(T_0)$ such that $t_0u_0$ is an element of $R_0$ comaximal to $J_0$.

    To show the converse, assume that for every $t_0\in T_0$ comaximal to $J_0$, there exists $u_0\in U(T_0)$ such that $t_0u_0$ is an element of $R_0$ comaximal to $J_0$. Let $f=t_0+t_1x+t_2x^2+\cdots\in T_\bNo[[x]]$ be comaximal to $J_\bNo[[x]]$. By Proposition \ref{unit construction}, this means that $t_0\in T_0$ is comaximal to $J_0$ and $t_i\in T_i$ for each $i\in\bN$. By the assumption, there must therefore exist some $u_0\in U(T_0)$ such that $t_0u_0$ is an element of $R_0$ comaximal to $J_0$. It will now suffice to show that there exist $b_i\in T_i$ for $i\in\bNo$ such that $f(u_0+b_1x+b_2x^2+\cdots)\in R_\bNo[[x]]$.

    Assume toward induction that for some $n\in\bN$, we have constructed $b_i\in T_i$ for $i<n$ such that for every $k<n$, $t_ku_0+t_{k-1}b_1+\cdots+t_0b_k\in R_k$. We need to show that we can now select $b_n\in T_n$ such that $t_nu_0+t_{n-1}b_1+\cdots+t_0b_n\in R_n$. Since $t_0\in T_0$ is comaximal to $J_0$, we can select $s_0\in T_0$ such that $s_0t_0\equiv 1\modulo{J_0}.$ Then we select $b_n=-s_0(t_nu_0+t_{n-1}b_1+\dots+t_1b_{n-1})\in T_n$ and observe that
    $$t_ku_0+t_{n-1}b_1+\dots+t_1b_{n-1}+t_0b_n=(1-t_0s_0)(t_ku_0+t_{n-1}b_1+\dots+t_1b_{n-1}).$$
    Since $(1-t_0s_0)\in J_0\subseteq (R_n:T_n)$, then $t_ku_0+t_{n-1}b_1+\dots+t_1b_{n-1}+t_0b_n\in R_n$. We can therefore continue selecting such $b_n$ inductively to construct an element $u=u_0+b_1x+b_2x^2+\cdots\in U(T_\bNo[[x]])$ such that $fu\in R_\bNo[[x]]$. Since $fu$ will also be comaximal to $J_\bNo[[x]]$ by construction of $u_0$, $R_\bNo[[x]]$ is a locally associated subring of $T_\bNo[[x]]$.
\end{proof}

\begin{corollary}
    Let $T$ be a commutative ring with identity, $R$ a subring with identity, and $J=(R:T)$. Then $R[[x]]$ is a locally associated subring of $T[[x]]$ if and only if $R$ is a locally associated subring of $T$.
\end{corollary}

Note that in Theorem \ref{la power series}, the necessary and sufficient condition on $R_0$ and $T_0$ is very similar to stating that $R_0$ is a locally associated subring of $T_0$. The only difference is that, rather than considering elements of $T_0$ and $R_0$ which are comaximal to the conductor $(R_0:T_0)$, we consider elements comaximal to $J_0\subseteq (R_0:T_0)$. As we will see in the examples that follow, this leaves open the possibility that $R_\bNo[[x]]$ may be locally associated in $T_\bNo[[x]]$ despite $R_0$ failing to be locally associated in $T_0$. However, the reverse is impossible, as the following corollary shows.

\begin{corollary}
    \label{la constants implies la ps}
    Let $\{R_i\}_{i=0}^\infty$ and $\{T_i\}_{i=0}^\infty$ be increasing sequences of commutative rings with identity such that $R_i\subseteq T_i$ for each $i\in\bNo$, and let $J_i=\bigcap_{i=0}^\infty(R_{i+k}:T_i)$ for each $i\in\bNo$. If $R_0$ is a locally associated subring of $T_0$, then $R_\bNo[[x]]$ is a locally associated subring of $T_\bNo[[x]]$.
\end{corollary}

\begin{proof}
    Let $t\in T_0$ be an element comaximal to $J_0$; that is, $tT_0+J_0=T_0$. Since $J_0\subseteq I:=(R_0:T_0)$, this also means that $t$ is comaximal to $I$. Since $R_0$ is a locally associated subring of $T_0$, there must exist some $u\in U(T_0)$ such that $ut$ is an element of $R_0$ comaximal to $I$. Let $r\in R$ and $s\in T$ such that $uts\equiv 1\modulo{J_0}$ and $utr\equiv 1\modulo{I}$. This also tells us that $uts\equiv 1\modulo{I}$, so $uts\equiv utr\modulo{I}$. Since $ut$ is a unit modulo $I$, $s\equiv r\modulo{I}$, and thus $s=r+\beta\in R_0$ for some $\beta\in I$. Therefore, $uts\equiv 1\modulo{J_0}$ implies that $utR_0+J_0=R_0$. Then $ut$ is an element of $R_0$ comaximal to $J_0$. By Theorem \ref{la power series}, $R_\bNo[[x]]$ is a locally associated subring of $T_\bNo[[x]]$.
\end{proof}

\begin{example}
    \label{ps not la}
    Let $y_0$ and $y_1$ be indeterminates, and let $T_0=\bZ[\frac{1}{2},y_0]$ and $R_0=\bZ+y_0\bZ+y_0^2T_0$. Note by Theorems \ref{associated polynomials} and \ref{la polynomial} that $R_0$ is an associated subring of $T_0$ which fails to be locally associated. Now let $R_1=T_0+3y_1T_0[y_1]$ and  $T_i=R_j=T_0[y_1]$ for $i\geq 1$ and $j\geq 2$. In this case, we have that $(R_0:T_0)=y_0^2T_0$; $(R_1:T_1)=3T_1$; and $(R_i:T_i)=T_1$ for all $i\geq 2$. Thus, $J_0=\bigcap_{i=0}^\infty (R_i:T_i)=3y_0^2T_0$. Similar to the construction in Example \ref{leading example}, we observe that $(2+3y_0)$ is comaximal to $J_0$ in $T_0$, since $(2+3y_0)(\frac{1}{2}-\frac{3}{4}y_0)=1-\frac{9}{4}y_0^2\equiv 1\modulo{3y_0^2}$. However, there is no unit $u\in U(T_0)=\{\pm 2^k|k\in\bZ\}$ such that $u(2+3y_0)$ is a unit in $R_0$ modulo $J_0$ (since $u$ would simultaneously need to satisfy $2u=\pm 1$ and $3u\in \bZ$). Then by Theorem \ref{la power series}, $R_\bNo[[x]]$ is not a locally associated subring of $R_\bNo[[x]]$. 
\end{example}

\begin{example}
    Let $\{y_i\}_{i=0}^\infty$ be a sequence of indeterminates and $\{p_i\}_{i=0}^\infty$ an enumeration of the prime numbers. Define $T_0=\bZ[p_0^{-1},y_0]$ and $R_0=\bZ+y_0\bZ+y_0^2T_0[y_0]$, and for each $i\in\bN$, define $T_i=T_{i-1}[p_i^{-1},y_i]$ and $R_i=T_{i-1}+y_iT_{i-1}+y_i^2T_i[y_i]$. Note that for each $i\in\bNo$, $(R_i:T_i)=y_i^2T_i$; thus, $J_0=\bigcap_{i=0}^\infty y_i^2T_i=\{0\}$ (since 0 is the only element of $T_0$ which is a multiple of $y_i$ for $i\in\bN$). Then by Theorem \ref{la power series}, $R_\bNo[[x]]$ is a locally associated subring of $T_\bNo[[x]]$. Interestingly enough, this holds despite the fact that $R_i$ fails to be locally associated in $T_i$ for every $i\in\bN_0$. This can easily be seen by observing that $(p_i+p_{i+1}y_i)\in T_i$ is comaximal to $(y_i)^2$ as an element of $T_i$, since $(p_i+p_{i+1}y_i)(\frac{1}{p_i}-\frac{p_{i+1}}{p_i^2}y_i)\equiv 1\modulo{y_i^2}$. However, there is no unit $u_i\in U(T_i)=\{\pm p_0^{a_0}\cdots p_i^{a_i}|a_i\in\bZ\}$ such that $u_i(p_i+p_{i+1}y_i)$ is an element of $R_i$ comaximal to $y_i^2T_i$, since $u_i$ would need to simultaneously satisfy $p_iu_i\in U(T_{i-1})$ and $p_{i+1}u_i\in T_{i-1}$.
\end{example}

\bibliographystyle{plain}
\bibliography{bibliography}

\end{document}